\documentclass{amsart}
\usepackage[margin=1.2in]{geometry}
\usepackage{amssymb}
\usepackage{amscd}
\usepackage[all]{xy}
\usepackage{bbm}
\usepackage{mathrsfs}
\usepackage{enumerate}
\usepackage{tikz}
\usepackage{stmaryrd}
\usepackage{etoolbox}
\usepackage{array}
\usepackage{verbatim}

\newcommand{\R}{\mathbb{R}}

\newcommand{\inv}{^{-1}}

\newcommand{\eps}{\varepsilon}

\newcommand{\del}{\nabla}

\newcommand{\lap}{\Delta}

\newcommand{\bd}{\partial}
\newcommand{\cl}{\overline}

\newcommand{\la}{\langle}
\newcommand{\ra}{\rangle}

\newcommand{\supp}{\operatorname{supp}}

\renewcommand{\div}{\operatorname{div}}

\newcommand{\grad}{\del}

\newcommand{\cp}{\operatorname{cap}}

\theoremstyle{plain}
\newtheorem{theorem}{Theorem}
\newtheorem{rethm}{Theorem} 

\newtheorem{prop}[theorem]{Proposition}
\newtheorem{lem}[theorem]{Lemma}
\newtheorem{conj}[theorem]{Conjecture}

\theoremstyle{definition}
\newtheorem{defn}[theorem]{Definition}
\newtheorem{rem}[theorem]{Remark}
\newtheorem{ex}[theorem]{Example}

\begin{document}
\title{Euclidean Domains with Nearly Maximal Yamabe Quotient} 
\author{Liam Mazurowski}
\author{Xuan Yao}

\begin{abstract}
     Let $\Omega$ be a smooth, bounded domain in $\R^3$ with connected boundary. It follows from work of Escobar that the Yamabe quotient of $\Omega$ is at most the Yamabe quotient of a  ball, and equality holds if and only if $\Omega$ is a ball. We show that if equality almost holds then the following things are true: (i)
      $\Omega$ is diffeomorphic to a ball;
      (ii) There is a small number $\eps > 0$ such that $B(x,r) \subset \Omega \subset B(x,r(1+\eps))$; (iii) After suitable scaling, $\Omega$ is Gromov-Hausdorff close to the unit ball when considered as a metric space with its induced length metric. 
      We also give a qualitative comparison between $Q$ and the coefficient of quasi-conformality studied in the theory of quasi-conformal maps. 
\end{abstract}

\maketitle



\section{Introduction} The Yamabe quotient $Q$ of a smooth manifold with boundary is a conformal invariant, studied by Escobar \cite{escobar1992yamabe} in the context of the Yamabe problem.  This invariant always satisfies $Q(M^n,[g]) \le Q(B^n,[g_{\text{euc}}])$, where $B$ is the unit ball in $\R^n$. Cherrier \cite{cherrier1984problemes} proved that if this inequality is strict, then $M$ admits a conformal metric with constant scalar curvature and minimal boundary. Escobar \cite{escobar1992yamabe} then verified that this inequality is indeed strict in many cases.  In particular, if $\Omega \subset \R^n$ is a smooth, bounded domain then 
\[
Q(\Omega) := Q(\Omega,[g_{\text{euc}}])
\] 
satisfies $Q(\Omega) \le Q(B)$ with equality if and only if $\Omega$ is a ball. 
Escobar \cite[Page 24]{escobar1992yamabe} asserted that one should ``think of $Q(\Omega)$ as measuring how far $\Omega$ is from a ball.''   This leads to a natural stability problem: if $Q(\Omega)$ is close to $Q(B)$ then in what sense is $\Omega$ close to a ball? 

In this paper, we investigate the stability problem for $Q$. Since the techniques we use ultimately depend on the Gauss-Bonnet theorem, we must restrict attention to domains in $\R^3$ with connected boundary. Our main results are as follows. 

\begin{theorem}
\label{main1}
For every $\eps > 0$ there is a $\delta > 0$ so that the following is true: whenever $\Omega\subset \R^3$ is a smooth, bounded domain with connected boundary and $Q(B) - Q(\Omega) < \delta$ then there exists a ball $B(x,r)$ such that $B(x,r) \subset \Omega \subset B(x,r(1+\eps))$.
\end{theorem} 

\begin{theorem} 
\label{main2}
For every $\eps > 0$ there is a $\delta > 0$ so that the following is true: whenever $\Omega\subset \R^3$ is a smooth, bounded domain with connected boundary and $Q(B) - Q(\Omega) < \delta$ then, after suitable scaling of $\Omega$, the Gromov-Hausdorff distance between the unit ball and $\Omega$ equipped with its induced length metric is at most $\eps$. 
\end{theorem}

\begin{theorem}
\label{main3}
    There is a $\delta > 0$ so that if $\Omega\subset \R^3$ is a smooth, bounded domain with connected boundary and $Q(B) - Q(\Omega) < \delta$ then $\Omega$ is diffeomorphic to a ball. 
\end{theorem}

\begin{rem}
Theorems \ref{main1}-\ref{main3} all remain true when the Yamabe quotient $Q(\Omega)$ is replaced by the Sobolev quotient $Q(\Omega,\bd \Omega)$ studied by Escobar in \cite{escobar1992conformal}. The proofs are almost identical. 
\end{rem} 

\begin{rem}
    Our methods show that in Theorems \ref{main1} and \ref{main2}, it suffices to take $\delta = O(\eps^9)$ as $\eps\to 0$. The constant in the big $O$ depends on the optimal constant in the De Lellis-M\"uller almost-umbilic estimate \cite{de2005optimal} and is therefore not explicit.
\end{rem}

\subsection{Motivation} 
The Riemann mapping theorem implies that conformal maps between domains in $\R^2$ exist in abundance.  In particular, every simply connected open set in $\R^2$ which is not the entire plane is conformal to the unit ball. On the other hand, conformal maps between domains in $\R^n$ are very sparse.  
Indeed, Liouville's theorem states that every conformal map defined on a domain $\Omega\subset \R^n$ is the restriction of a M\"obius transformation. Therefore, the only open subsets of $\R^n$ conformal to the unit ball are balls and half-spaces. 

 In $\R^n$, the class of quasi-conformal maps exhibits greater flexibility, more akin to the picture in two dimensions. Informally, a quasi-conformal map $f$ is a homeomorphism that sends infinitesimal circles to infinitesimal ellipses with eccentricity bounded by some constant $K < \infty$.  The smallest such constant $K$ is called the dilation of $f$ and is denoted by $K(f)$.  One always has $K(f)\ge 1$ and equality holds if and only if $f$ is conformal.  
 

Given a domain $\Omega\subset \R^n$, the {\it coefficient of quasi-conformality} is a conformal invariant $K(\Omega)$ which measures how far $\Omega$ is from being conformal to a ball.  This invariant was introduced by Gehring \cite{gehring1962rings}. 
More formally, the coefficient of quasi-confromality is defined by
\[
K(\Omega) =  \inf\{K(f): f\colon B\to \Omega \text{ is quasi-conformal}\},
\]
where by convention $K(\Omega) = \infty$ if there is no quasi-conformal map from $B$ to $\Omega$. In particular $K(\Omega) = \infty$ if $\Omega$ is not homeomorphic to a ball.  When $K(\Omega) < \infty$, it turns out \cite{gehring1962rings} that there is always an optimal quasi-conformal map $f \colon B\to \Omega$ whose dilation satisfies $K(f) = K(\Omega)$.  Thus $K(\Omega) \ge 1$ with equality if and only if $\Omega$ is conformal to a ball.  
This leads to a natural stability problem: if $K(\Omega)$ is close to 1 then in what sense is $\Omega$ close to a ball? 
Gehring and V\"ais\"al\"a \cite{gehring1965coefficients} have studied the coefficient of quasi-conformality of domains $\Omega\subset \R^3$ in great detail. While they do not give a complete characterization of domains for which $K(\Omega)$ is close to 1, their work shows that certain geometric features can prevent $K(\Omega)$ from being close to 1.

Given that the conformal invariants $K(\Omega)$ and $Q(\Omega)$ both measure in a certain sense how far $\Omega$ is from a ball, it is natural to seek comparisons between these two invariants.  Is one stronger than the other? Or do they detect essentially the same features of the domain $\Omega$? Along these lines, Escobar \cite[Page 24]{escobar1992yamabe} proposed to ``study the relation between $Q(\Omega)$ and the dilation quotients defined in the study of quasi-conformal maps.'' Our results allow for a qualitative comparison between $K$ and $Q$ in the regime where $K$ is close to 1 and $Q$ is close to $Q(B)$. We will show that certain geometric features which force $K(\Omega)$ to be far from 1 also force $Q(\Omega)$ to be far from $Q(B)$. 

\begin{rem}
From the point of view of comparing $K$ and $Q$, it is reasonable to require a priori that $\Omega$ has connected boundary (or even that $\Omega$ is diffeomorphic to a ball). Indeed, every quasi-conformal map is required by definition to be a homeomorphism.
\end{rem}

\subsection{Further Problems}

Given a smooth, compact manifold $M^3$ with boundary, the relative $\sigma$-constant of $M$ is defined by 
\[
\sigma(M) = \sup_g Q(M,[g]),
\]
where the supremum is taken over all smooth metrics $g$ on $M$. This is an invariant of the smooth manifold with boundary. One always has $\sigma(M) \le \sigma(B)$ where $B$ is a ball. 

Handlebodies are the simplest compact 3-manifolds with boundary.  Schwartz's connect sum inequality for the Yamabe quotient of a manifold with boundary \cite{schwartz2009monotonicity} implies that $\sigma(H) = \sigma(B)$ for all handlebodies $H$. Thus $\sigma$ cannot distinguish between a ball and a handlebody. On the other hand, if $M^3$ admits a smooth embedding into $\R^3$, then we can define a modified invariant 
\[
\sigma_{\text{euc}}(M) = \sup_{\Omega} Q(\Omega),
\]
where the supremum is taken over all smooth, bounded domains $\Omega\subset \R^3$ whose closure is diffeomorphic to $M$. This modified invariant is defined for all handlebodies, and Theorem \ref{main3} implies that $\sigma_{\text{euc}}(H) < \sigma_{\text{euc}}(B)$ whenever $H$ is a handlebody with at least one handle.

\begin{conj}
Let $T$ be a solid torus, i.e., a handlebody with one handle. We conjecture that $\sigma_{\operatorname{euc}}(T)$ is attained by the  region $C$ enclosed by the image of a Clifford torus under stereographic projection. 
\end{conj}

We remark that already the value $Q(C)$ does not seem to be known. The Yamabe quotient $Q(C)$ is {\it not} achieved by the round metric. 

\subsection{The Yamabe Quotient} 

The Yamabe problem asks whether every closed Riemannian manifold $(M^n,g)$ admits a conformal metric with constant scalar curvature. Yamabe \cite{yamabe1960deformation} showed that it is equivalent to find a minimizer of the functional 
\[
Q_g(f) =  \frac{\int_M  \vert \grad f\vert^2 + \frac{n-2}{4(n-1)} Rf^2\, dv}{\left(\int_M f^{\frac{2n}{n-2}}\, dv\right)^{\frac{n-2}{2n}}},
\]
and then argued that such a minimizer must always exist. 
However, Trudinger \cite{trudinger1968remarks} later discovered a gap in Yamabe's proof. Trudinger was able to repair the gap when the conformal invariant
\[
Q(M,[g]) = \inf_{f} Q_g(f)
\]
satisfies $Q(M,[g]) \le 0$. Aubin \cite{aubin1976equations} proved it is actually enough to have $Q(M,[g]) < Q(S^n,[g_{\text{round}}])$, and verified that this inequality holds in many situations. Finally, Schoen \cite{schoen1984conformal} proved the strict inequality in all remaining cases, leading to a complete resolution of the Yamabe problem. 

There is a natural variant of the Yamabe problem on manifolds with boundary. Given a smooth Riemannian manifold $(M,g)$ with boundary, one seeks a conformal metric with constant scalar curvature in the interior and zero mean curvature on the boundary. In this case, the associated functional is given by
\[
Q_g(f) = \frac{\int_M  \vert \grad f\vert^2 + \frac{n-2}{4(n-1)} Rf^2\, dv +  \int_{\bd M} \frac{n-2}{2(n-1)} H u^2\, da}{\left(\int_M f^{\frac{2n}{n-2}}\, dv\right)^{\frac{n-2}{2n}}}.
\]
One again defines the Yamabe quotient 
\[
Q(M,[g]) = \inf_f Q_g(f),
\]
which is an invariant of the conformal class. Cherrier \cite{cherrier1984problemes} showed that if $Q(M,[g]) < Q(B,[g_{\text{euc}}])$ then there exists a conformal metric with constant scalar curvature and minimal boundary. Escobar \cite{escobar1992yamabe} verified the inequality $Q(M,[g]) < Q(B,[g_{\text{euc}}])$ in many cases. In particular, given a smooth bounded domain $\Omega\subset \R^n$, let $Q(\Omega) = Q(\Omega,[g_{\text{euc}}])$. Then Escobar's work implies that 
\begin{equation}
\label{esc1} 
Q(\Omega)\le Q(B)
\end{equation}
with equality if and only if $\Omega$ is a ball. 

In fact, on a manifold with boundary, one can also look for a conformal metric with zero scalar curvature in the interior and constant mean curvature on the boundary. The basic strategy is very similar, except that one replaces $Q_g$ by the functional 
\[
Q_g^{\bd}(f)  = \frac{\int_M  \vert \grad f\vert^2 + \frac{n-2}{4(n-1)} Rf^2\, dv +  \int_{\bd M} \frac{n-2}{2(n-1)} H u^2\, da}{\left(\int_{\bd M} f^{\frac{2(n-1)}{n-2}}\, dv\right)^{\frac{n-2}{2(n-1)}}}.
\]
Again the problem is solvable provided the conformal invariant 
\[
Q(M,\bd M,[g]) = \inf_f Q_g^{\bd}(f)
\]
satisfies $Q(M,\bd M,[g]) < Q(B,\bd B,[g_{\text{euc}}])$. Escobar \cite{escobar1992conformal}  showed that this strict inequality holds in many cases. In particular, given a smooth bounded domain $\Omega\subset \R^n$, let $Q(\Omega,\bd \Omega) = Q(\Omega,\bd \Omega, [g_{\text{euc}}])$. Then Escobar's work implies that 
\begin{equation}
\label{esc2}
Q(\Omega,\bd \Omega)\le Q(B,\bd B)
\end{equation} 
with equality if and only if $\Omega$ is a ball.

\subsection{Level Set Methods} 
Escobar's proof of rigidity in inequality (\ref{esc1}) depends on a positive mass type theorem for manifolds with boundary. Recently,  new monotonicity formulas for harmonic functions have been discovered, and these can be used to give new proofs of the positive mass theorem. Moreover, these harmonic function techniques are well-suited to proving stability theorems. In this subsection, we briefly survey these developments. 

In non-negative Ricci curvature, Colding \cite{colding2012new} discovered that the Bochner formula implies the monotonicity of several quantities along the level sets of a harmonic function. As an application, Colding deduced a sharp gradient estimate for the Green's function on a manifold with non-negative Ricci curvature.  See the references \cite{agostiniani2020sharp,agostiniani2020monotonicity,benatti2021minkowski,colding2014ricci} for further applications of the level set method for harmonic functions in Euclidean space and manifolds with non-negative Ricci curvature. 

Stern \cite{stern2022scalar} discovered a rearrangement of the Bochner formula, enabling the study of scalar curvature via the level sets of harmonic functions; also see the work of Kijowski \cite{kijowski1984unconstrained} and Jezierski-Kijowski \cite{jezierski2005unconstrained}.  Bray-Kazaras-Khuri-Stern \cite{bray2022harmonic} used Stern's method to give a new proof of the positive mass theorem. Since then, many monotone quantities have been discovered along the level sets of harmonic functions in manifolds with non-negative scalar curvature. These have been applied to study the Green's function on manifolds with non-negative scalar curvature \cite{munteanu2021comparison}, give new proofs of the positive mass theorem \cite{agostiniani2024green, miao2023mass}, study the relationship between mass and capacity \cite{miao2023mass,oronzio2022adm}, and re-compute the Yamabe invariant of $\mathbb{RP}^3$ \cite{mazurowski2023yamabe}.  See also \cite{agostiniani2022riemannian,chan2024monotonicity,hirsch2024monotone,mazurowski2023monotone,mazurowski2024mass, xia2024new} for generalizations to $p$-harmonic functions. 

Harmonic function techniques have proven especially useful for studying stability theorems.  For example, the positive mass theorem \cite{schoen1979proof,witten1981new} says that any asymptotically flat manifold with non-negative scalar curvature has non-negative ADM mass. Moreover, if the mass is zero then the manifold must be Euclidean space. The stability problem for the positive mass theorem asks whether an asymptotically flat manifold with very small mass must be close to Euclidean space in some sense. 
Dong \cite{dong2024some} and Kazaras-Khuri-Lee \cite{kazaras2021stability} were able to prove partial results in this direction under certain extra assumptions.  Later, Dong and Song \cite{dong2023stability} proved a definitive stability theorem for the positive mass theorem, confirming a conjecture of Huisken and Ilmanen \cite{huisken2001inverse}. All of these results depend crucially on \cite{bray2022harmonic}. Harmonic functions have also been very useful for studying the stability of other inequalities in scalar curvature. For example, Dong \cite{dong2024stability} has proved a stability theorem for the Penrose inequality.   Likewise, Allen-Bryden-Kazaras \cite{allen2023stability} have proven a stability result for Llahrul's theorem in dimension 3; also see Hirsch-Zhang \cite{hirsch2024stability} for a higher dimensional version. 
Finally, we mention that the authors \cite{mazurowski2024stability} have used harmonic functions to study the stability of the Yamabe invariant of $S^3$.

\subsection{Sketch of Proof} 
In this subsection, we sketch the proof of Theorem \ref{main1}. The proof has two major steps. 
In the first step, we prove a preliminary stability result for the Yamabe quotient of {\it exterior} domains.  Then in the second step, we show that a careful application of the preliminary stability result leads to Theorem \ref{main1}. The proofs of Theorems \ref{main2} and  \ref{main3} are very similar.

{\bf Step 1:} It will be convenient to consider the Yamabe quotient of exterior domains. More precisely, given a smooth, bounded domain $\Omega \subset \R^3$, define 
\[
Q^*(\Omega) =  \inf_f\frac{\int_{\R^3\setminus \Omega} \vert \grad f\vert^2 - \frac{1}{4}\int_{\bd \Omega} Hf^2\, da}{\left(\int_{\R^3\setminus \Omega} f^6\, dv\right)^{1/3}}. 
\]
Here our convention is that $H = 2$ when $\Omega$ is the unit ball, and the infimum is taken over all functions $f\in W^{1,2}(\R^3\setminus \Omega)$. 
Let  $\Omega^*$ be the interior of $\R^3\setminus [\phi(\Omega)\setminus\{0\}]$, where $\phi(x) = x/\vert x\vert^2$ is the conformal inversion about the origin. Since the Yamabe quotient is conformally invariant, we have $Q^*(\Omega) = Q(\Omega^*)$. In particular, we have 
\begin{equation}
\label{esc3}
Q^*(\Omega) \le Q^*(B)
\end{equation}
with equality if and only if $\Omega$ is a ball.  

Let $\Omega\subset \R^3$ be a smooth, bounded domain with connected boundary. We are going to prove the following preliminary stability result: 

\begin{prop}
\label{prelim-stability}
For every $\eps > 0$ there is a $\delta > 0$ such that if $Q^*(B) - Q^*(\Omega) < \delta$ then there is a ball $B(x,r)$ such that $\Omega \subset B(x,r)$ and $\bd B(x,r)$ is contained in the $\eps r$-neighborhood of $\bd \Omega$. 
\end{prop} 

Escobar's original proof of rigidity in (\ref{esc3}) involves a conformal change by a harmonic function $u$ which satisfies a Neumann type boundary condition $\bd_\nu u = \frac{1}{4}H u$.  Since $\bd \Omega$ can be arbitrarily wild a priori and the mean curvature of $\bd \Omega$ appears in this boundary condition, it seems very difficult to control such a function.  Therefore, we re-derive the inequality (\ref{esc3}) with rigidity using the capacitary potential of $\Omega$. More precisely, let $u$ be the solution to 
\[
\begin{cases}
\lap u = 0, &\text{in } \R^3\setminus \Omega,\\
u = 1, &\text{on } \bd \Omega,\\
u\to 0, &\text{at infinity}. 
\end{cases}
\]
Then let $w = -\log u$.  We show that there is a test function of the form $f = \vert \grad w\vert^{1/2} s(w)$ which satisfies 
\begin{equation}
\label{f-est}
\frac{\int_{\R^3\setminus \Omega} \vert \grad f\vert^2 - \frac{1}{4}\int_{\bd \Omega} Hf^2\, da}{\left(\int_{\R^3\setminus \Omega} f^6\, dv\right)^{1/3}}\le Q^*(B)
\end{equation}
with equality if and only if $\Omega$ is a ball. To prove this, we combine several monotonicity formulas of Miao \cite{miao2023mass} together with a novel integration by parts argument.  It is essential for our applications that $u$ satisfies a {Dirichlet} type condition on $\bd \Omega$. 

Next we focus on the stability. In fact, using the monotonicity formulas, we can refine the estimate (\ref{f-est}) to show that 
\[
\int_0^\infty b(\tau) \int_{\{w=\tau\}} \| \mathring A\|^2\, da \, d\tau \le Q^*(B) - Q^*(\Omega),
\]
where $b$ is a smooth function with $b(\tau) > 0$ for $\tau > 0$. 
Write $o(1)$ for a quantity that goes to 0 as $Q^*(\Omega)\to Q^*(B)$. It follows from the above inequality that some level set $\Sigma = \{u = 1-o(1)\}$ satisfies 
\[
\int_{\Sigma} \| \mathring A\|^2\, da = o(1). 
\]
We can then apply the almost-umbillic estimate of De Lellis-M\"uller \cite{de2005optimal} to deduce that $\Sigma$ is very close to a round sphere $\bd B(x,r)$. To conclude, we need to compare $\bd \Omega$ with $\bd B(x,r)$. The fact that $\Sigma$ is a level set of the capacitary potential implies that 
\[
\frac{\cp( \Omega)}{\cp( B(x,r))} \ge 1 - o(1). 
\] 
Finally, we show that the only way this can occur is if every point of $\bd B(x,r)$ is close to some point of $\bd \Omega$, as in the following figure.
\[
\includegraphics[width=2in]{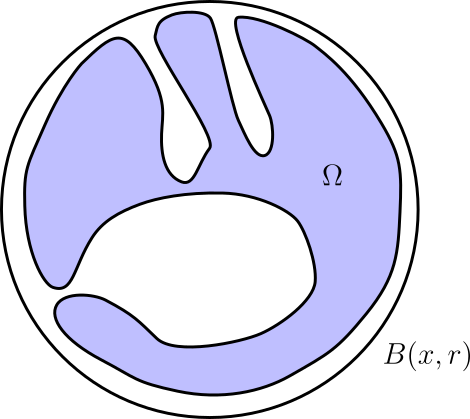}
\]
The preliminary stability result then follows.

{\bf Step 2:} Now we return to the problem for interior domains. Let $\Omega\subset \R^3$ be a smooth, bounded domain with connected boundary. We would like to apply the preliminary stability result to deduce Theorem \ref{main1}. To motivate our ideas, it is helpful to consider the following example. 

\begin{ex}
Let $\Omega_k$ be a dumbbell shaped domain obtained by smoothing 
\[
B((-2,0,0),1) \cup B((2,0,0),1) \cup \left\{y^2 + z^2\le \frac{1}{k^2},\ -2\le x\le 2\right\}.
\]
We are going to show that $Q(\Omega_k) \not\to Q(B)$ as $k\to \infty$. In particular, this shows that the Yamabe connect sum inequality \cite{schwartz2009monotonicity} cannot be achieved geometrically within $\R^3$. 

Suppose to the contrary that $Q(\Omega_k)\to Q(B)$.  Our strategy is to conformally invert about a point $p\in \Omega_k$ and then appeal to Proposition \ref{prelim-stability} to get a contradiction. For this to work, it is essential to make a good choice for the point $p$. Indeed, suppose we pick $p = (-2,0,0)$ and then let 
\[
\phi(x) = \frac{x  + (2,0,0)}{\vert x + (2,0,0)\vert^2}. 
\]
Let $\Omega_k^*$ be the interior of $\R^3\setminus [\phi(\Omega_k\setminus \{(-2,0,0)\})]$.  Then $Q^*(\Omega_k^*) \to Q^*(B)$ so we can apply Proposition \ref{prelim-stability}. However, this does not lead to a contradiction. Indeed, the set $\Omega_k^*$ looks like a ball with a thin neck drilled out leading to a large cavern. In particular, $\bd B(0,1)$ is contained in a small neighborhood of $\bd \Omega_k^*$.  The problem is that, for this choice of $p$, the conformal inversion ``hides away'' the features that distinguish $\Omega_k$ from a ball. 
\[
\includegraphics[width=4in]{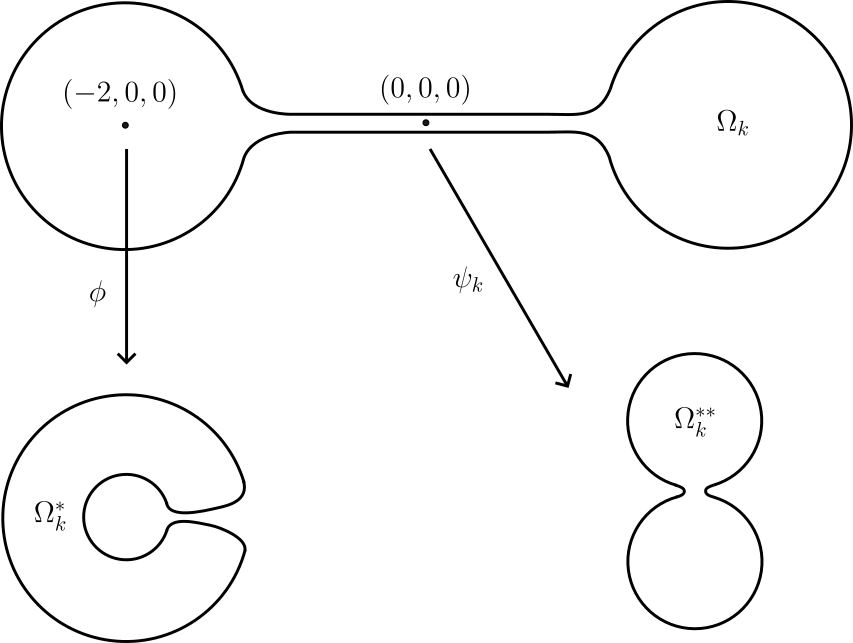}
\]
In this example, a better choice is to take $p = (0,0,0)$. Let 
\[
\psi_k(x) = \frac{x}{k\vert x\vert^2}
\]
be the conformal inversion centered at the origin which takes $\bd B(0,k\inv)$ to $\bd B(0,1)$. Let $\Omega_k^{**}$ be the interior of $\R^3 \setminus [\psi_k(\Omega_k\setminus \{(0,0,0)\})]$. Then again we have $Q^*(\Omega_k^{**}) \to Q^*(B)$ and so we can apply Proposition \ref{prelim-stability}. In this case, we do get a contradiction. Indeed, when $k$ is very large, the set $\Omega_k^{**}$ is essentially the interior of a horn torus, and it is clear that the conclusion of Proposition \ref{prelim-stability} does not hold. 
\end{ex} 

In the general case, we need to find a systematic way to select the point $p$. Motivated by the previous example, we will select $p$ to be the center of the smallest ball contained in $\Omega$ which ``intersects $\bd \Omega$ in multiple different directions.'' More precisely, let  
$\mathcal A(\Omega)$ be the set of all balls $B(x,r)\subset \Omega$ such that there exist points $y,z\in \bd B(x,r)\cap \bd \Omega$ for which the vectors $y-x$ and $z-x$ make at least a 45 degree angle. The choice of 45 degrees is arbitrary and we have selected it purely for convenience; any other fixed positive angle would suffice. The set $\mathcal A(\Omega)$ is always non-empty. Given $B(x,r)\in \mathcal A(\Omega)$, we can do a conformal inversion centered at $x$, apply Step 1, and then undo the conformal inversion to deduce the following structural result:

\begin{prop}
\label{A-structure}
For every $\eps > 0$ there is a $\delta > 0$ so that if $Q(B)-Q(\Omega) < \delta$ then the following property holds: whenever $B(x,r)\in \mathcal A(\Omega)$ it follows that $\bd B(x,r)$ is contained in the $\eps r$-neighborhood of $\bd \Omega$.  
\end{prop}

Finally, we select $B(p,r)$ to be a ball in $\mathcal A(\Omega)$ of smallest radius. Such a ball must exist because $\Omega$ is smooth. Applying the previous result, we deduce there is a dense net of points belonging to the complement of $\Omega$ lying just outside $\bd B(p,r)$. 
The following figure shows $\Omega$ together with the ball $B(p,r) \in \mathcal A(\Omega)$. The cross marks indicate points belonging to the complement of $\Omega$.  To prove Theorem \ref{main1}, we want to show that $\Omega \subset B(p,r(1+\eps))$. Suppose to the contrary that there is a point $q\in \Omega$ with $\vert q-p\vert \ge r(1+\eps)$ as in the figure. Then we need to argue that, as indicated by the blue circles below,  there is actually a ball in $\mathcal A(\Omega)$ with radius smaller than $r$.

\[
\includegraphics[width=3in]{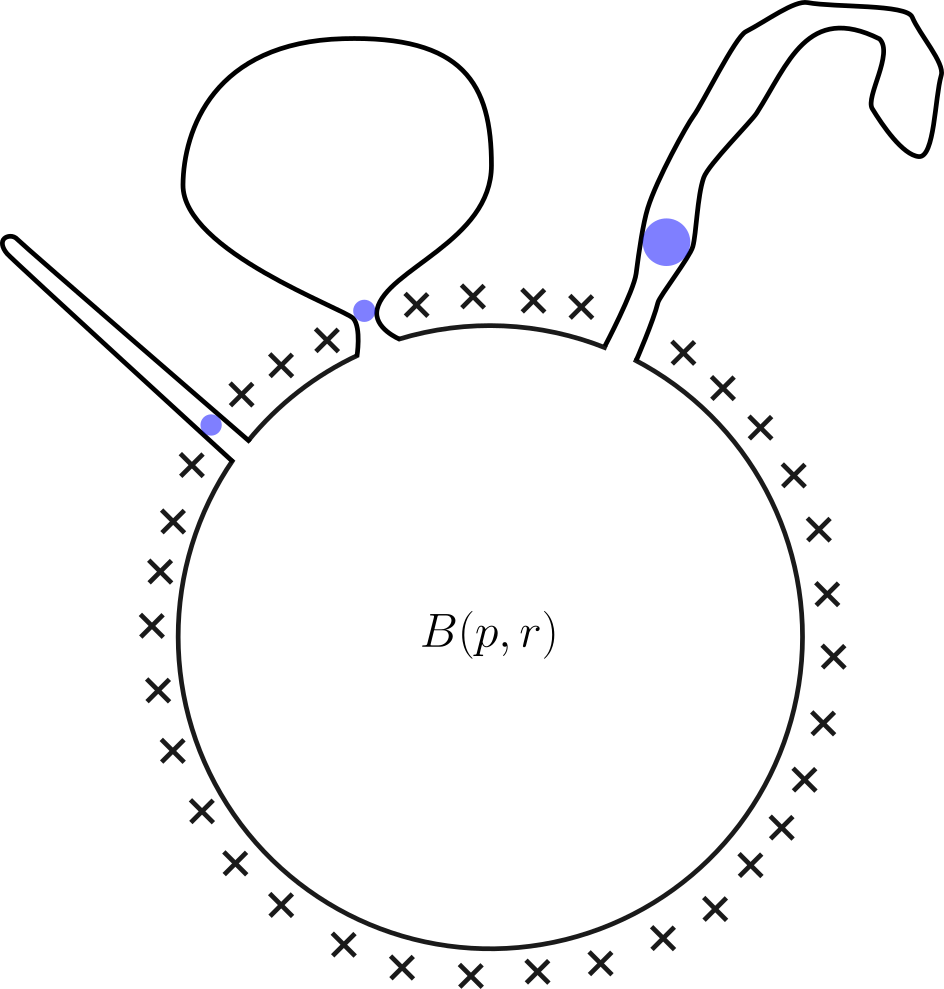}
\]

While it is intuitively clear that such balls must exist, it seems somewhat subtle to rigorously prove this. To do so, we rely on some tools from computational geometry developed to study the medial axis. The {\it medial axis} $\mathcal M$ of $\Omega$ is the set of all points $x\in \Omega$ with more than one nearest neighbor on $\bd \Omega$.  The radius function $\mathcal R$ is defined by $\mathcal R(x) = d(x,\bd \Omega)$ for $x\in \Omega$.  Finally, the separation angle at $x\in \mathcal M$ is the largest possible angle that can be formed by the vectors $y-x$ and $z-x$ for $y,z \in \bd B(x,\mathcal R(x)) \cap \bd \Omega$. In this language, to complete the proof, we need to find a point $x\in \mathcal M$ where $\mathcal R(x) < r$ and where the separation angle is at least 45 degrees.  This can be done using a flow $\mathfrak C$ developed by Lieutier \cite{lieutier2004any}. This flow $\mathfrak C$ extends the gradient flow of $\mathcal R$. Lieutier originally used it to show that every bounded open set in Euclidean space is homotopic to its medial axis.  For our purposes, there is a key relation between $\mathcal R$ and the separation angle along this flow, and this makes it perfectly suited for finding the desired point $x$. 

\subsection{Organization} 

The remainder of this paper is organized as follows. In Section \ref{Section:Medial}, we discuss some background material related to the medial axis.  In Section \ref{Section:Harmonic} we give a new proof of Escobar's inequalities using the capacitary potential. The main results are proven in Section \ref{Section:Stability}. Finally, in Section \ref{Section:Comparison}, we give a qualitative comparison between the Yamabe quotient $Q$ and the coefficient of quasi-conformality $K$.  

\subsection{Acknowledgements} The authors would like to thank Xin Zhou for his continual support and guidance. L.M. is supported by an AMS Simons Travel Grant. X.Y. is supported by Hutchinson Fellowship.

\section{The Medial Axis} 
\label{Section:Medial}

As part of our proof, we rely on some tools from computational geometry that may be unfamiliar to the reader. In particular, we will make use of the medial axis of an open subset of $\R^3$. In this brief section, we review the relevant background material.  The medial axis was first introduced by Blum \cite{blum1967transformation}, as a tool to study the shape of biological structures. For example, the medial axis of a network of blood vessels acts as a sort of centerline, capturing the overall shape of the network. 

Mathematically, the medial axis of an open set $\Omega$ is the set of points with more than one nearest neighbor on the boundary. Lieutier \cite{lieutier2004any} proved that any open bounded subset of $\R^n$ is homotopy equivalent to its medial axis. In the following, we use the notation from \cite{lieutier2004any}. Let $\Omega\subset \R^n$ be an open bounded set.

\begin{defn} Given $x\in \Omega$, define the set  
$
\Gamma(x) = \{y\in \bd \Omega: d(x,y) = d(x,\bd \Omega)\}. 
$
Thus $\Gamma(x)$ is the set of points on $\bd \Omega$ which are closest to $x$. 
\end{defn}

\begin{defn} The {\it medial axis} of $\Omega$ is the set 
$
\mathcal M = \{x\in \Omega: \# \Gamma(x) \ge 2\}.
$
\end{defn} 

The following figure illustrates the medial axis of an hour-glass shaped polygon in the plane. 
\[
\includegraphics[width=2in]{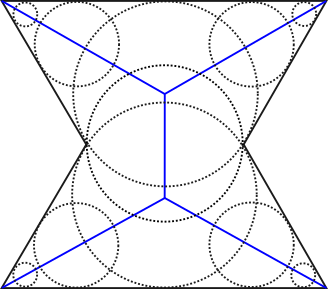}
\]
When $x\in \mathcal M$, there exists a unique closed ball of smallest radius containing $\Gamma(x)$. Following \cite{lieutier2004any}, we denote the center of this ball by $\Theta(x)$ and the radius of this ball by $\mathcal F(x)$. When $ \#\Gamma(x) = 1$, we let $\Theta(x)$ be the unique point of $\bd \Omega$ closest to $x$ and we set $\mathcal F(x) = 0$.  

Define the radius function $\mathcal R(x) = d(x,\bd \Omega)$ for $x\in \Omega$. 
Then, for each $x\in \Omega$, define the vector 
\[
\nabla(x) = \frac{x-\Theta(x)}{\mathcal R(x)}.
\]
This vector coincides with the gradient of $\mathcal R$ when $x \in \Omega \setminus \overline{\mathcal  M}$. When $x\in \cl{\mathcal M}$, the vector $\nabla (x)$ still points in the direction in which $\mathcal R$ increases the fastest. It is proven in \cite{lieutier2004any} that 
\[
\|\nabla(x)\|^2 = 1 - \frac{\mathcal F(x)^2}{\mathcal R(x)^2}.
\]
Therefore one always has $0\le \|\nabla(x)\|\le 1$, and $\|\nabla(x)\| < 1$ precisely when $x$ belongs to the medial axis. The function $\mathcal F(x)$ is upper semi-continuous, and $\mathcal R(x)$ is continuous, and therefore $\|\nabla(x)\|$ is lower semi-continuous. 

\begin{defn}
Fix a point $x\in \mathcal M$. Given two points $y,z\in  \Gamma(x)$, let $\angle y x z \in [0,\pi]$ denote the angle formed by the vectors $y-x$ and $z-x$.  The {\it separation angle} at the point $x$ is defined to be 
$
\max \{\angle y x z: y,z\in \Gamma(x)\}.
$
\end{defn}

\begin{prop}
\label{angle}
Fix a point $x\in \mathcal M$. If $\|\nabla(x)\|^2 \le \frac 1 2$ then the separation angle at $x$ is at least $\frac{\pi}{4}$. 
\end{prop}

\begin{proof}
We will prove the contrapositive. Suppose the separation angle $\alpha$ at $x$ is less than $\frac{\pi}{4}$.  Choose a point $y\in \Gamma(x)$ and then let $C$ be the closed cone with vertex $x$, axis $y-x$, and opening angle $\alpha$. Then it follows that $\Gamma(x) \subset C \cap \bd B(x,\mathcal R(x))$. By elementary geometry, there is a closed ball of radius $r = \mathcal R(x) \sin(\alpha)$ which encloses $C \cap  \bd B(x,\mathcal R(x))$. It follows that $\mathcal F(x) \le \mathcal R(x) \sin(\alpha)$ and therefore that $\|\nabla(x)\|^2 > \frac{1}{2}$. 
\end{proof} 

In general, the vector field $\nabla$ will not depend continuously on $x$. Nevertheless, it is proved in \cite{lieutier2004any} that there is a well-defined, continuous map 
$
\mathfrak C\colon  [0,\infty)\times \Omega \to \Omega
$
which ``is the flow of'' $\nabla$. More precisely one has $\mathfrak C(s,\mathfrak C(t,x)) = \mathfrak C(s+t,x)$ and 
\begin{gather*}
\mathfrak C(t,x) = x + \int_0^t \del(\mathfrak C(\tau,x))\, d\tau,\\
\mathcal R(\mathfrak C(t,x)) = \mathcal R(x) + \int_0^t \|\del (\mathfrak C(\tau,x))\|^2\, d\tau. 
\end{gather*}
The flow is only defined forward in time and it may happen that $\mathfrak C(t,x) = \mathfrak C(t,y)$ even though $x\neq y$. Trajectories of $\mathfrak C$ are continuous, but may have sharp corners. The following figure illustrates the flow map $\mathfrak C$ for an hour-glass shaped polygon.
\[
\includegraphics[width=4in]{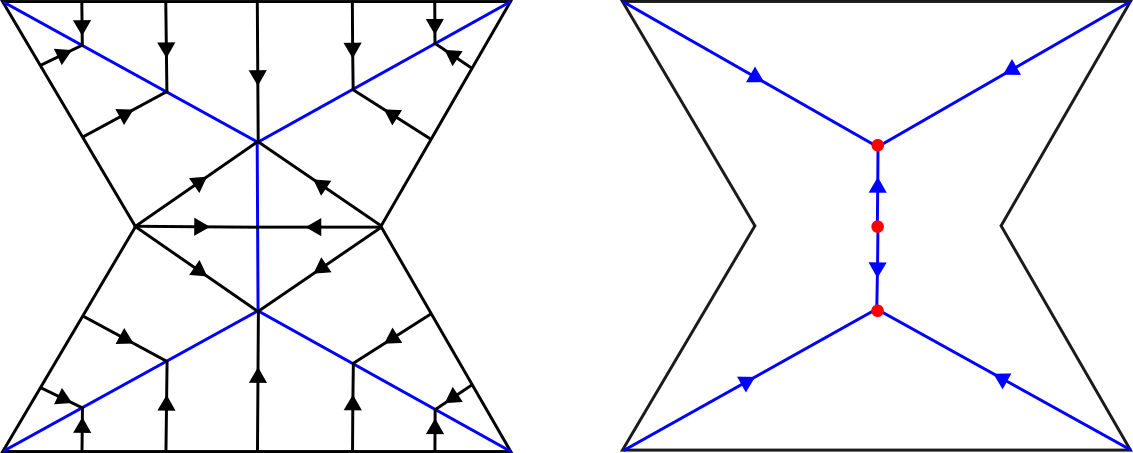}
\]
Points which do not lie on the medial axis flow directly away from the closest point on the boundary until they hit the medial axis. Points of the medial axis flow along the medial axis in the direction in which the radius $\mathcal R$ increases the fastest. There are three fixed points indicated by the red dots. These fixed points satisfy $\Theta(x) = x$ and $\mathcal F(x) = \mathcal R(x)$ so that $\nabla (x) = 0$.

\section{Escobar's Inequalities}

\label{Section:Harmonic}

Let $\Omega \subset \R^3$ be a smooth, bounded domain with connected boundary. Let $u$ be the capacitary potential for $\Omega$, i.e., the solution to 
\[
\begin{cases}
    \lap u = 0, &\text{in } \R^3\setminus \Omega\\
    u=1, &\text{on } \bd \Omega\\
    u\to 0, &\text{at infinity}.
\end{cases}
\]
Then set $w = -\log u$. We note that $w$ satisfies $\lap w = \vert \grad w\vert^2$.
In this section, we show how to construct a good test function for the Yamabe quotient of the exterior domain $\R^3\setminus \Omega$ using the function $w$. 

\begin{defn}
    Define the test function 
    \[
    f = s(w) \vert \grad w\vert^{1/2}, \quad \text{where } s(t) = \sqrt{\frac{e^w}{1+e^{2w}}}.
    \]
\end{defn}

\begin{rem} In the model case where $\Omega = B$ is the unit ball, the optimal test function for the Yamabe quotient is the conformal factor transforming $\R^3\setminus \Omega$ into a round hemisphere. Note that in this model situation $u = r^{-1}$ and $w=\log r$. Thus
$f$ is the optimal test function in the model setting. 
\end{rem}


\subsection{Rigidity} First we concentrate on obtaining inequality (\ref{esc3}) with rigidity. Supposing $t$ is a regular value of $w$, define $\Sigma_t = \{w=t\}$. 
The assumption that $\bd \Omega$ is connected ensures that $\Sigma_t$ is connected for all regular values $t$ (see \cite[Lemma 6]{mazurowski2024stability}). 
Define the following quantities:
\[
W(t):=\int_{\Sigma_t}|\nabla w|^2,\quad U(t):=\int_{\Sigma_t}H|\nabla w|.
\]

\begin{lem}
For each regular value $t$ of $w$, the above quantities are differentiable and satisfy:
\begin{align*}
    W'(t)=2W(t)-U(t),
\end{align*}
and
\begin{align*}
    U'(t)\leq 8\pi-U(t)+W(t)-\frac{|\nabla|\nabla w||^2}{|\nabla w|^2}.
\end{align*}
The latter inequality becomes an equality when $\Omega$ is a ball. 
\end{lem}

\begin{proof}
    The first equality follows from direct computations. We focus on the second inequality. Note that
    \[
    H=|\nabla w|-\frac{\langle \nabla |\nabla w|,\nabla w\rangle}{|\nabla w|^2}.
    \]
    We compute
    \begin{align*}
        U'(t)&=\int_{\Sigma_t}H^2+\frac{\partial H}{\partial t}|\nabla w|+H\frac{\langle \nabla |\nabla w|,\nabla w\rangle}{|\nabla w|^2},\\
             &=\int_{\Sigma_t}-\|A\|^2+H|\nabla w|-\frac{|\nabla^{\Sigma}|\nabla w||^2}{|\nabla w|^2}\\
             &\leq 8\pi -\int_{\Sigma_t}H^2+\int_{\Sigma_t}H|\nabla w|-\int_{\Sigma_t}\frac{|\nabla|\nabla w||^2}{|\nabla w|^2}+\int_{\Sigma_t}\frac{\langle \nabla |\nabla w|,\nabla w\rangle^2}{|\nabla w|^4}.
    \end{align*}
    In the last step, we used the Gauss equation, 
    \[
    0=2K_{\Sigma_t}+\|A\|^2-H^2.
    \]
We rewrite the terms as
\begin{align*}
    &-H^2+H|\nabla w|+\frac{|\nabla^{\perp}|\nabla w||^2}{|\nabla w|^2}=\\
    &-\left(|\nabla w|^2-2\frac{\langle\nabla|\nabla w|,\nabla w\rangle}{|\nabla w|}+\frac{|\nabla^{\perp}|\nabla w||^2}{|\nabla w|^2}\right)+|\nabla w|^2-\frac{\langle\nabla|\nabla w|,\nabla w\rangle}{|\nabla w|}+\frac{|\nabla^{\perp}|\nabla w||^2}{|\nabla w|^2},\\
    &=\frac{\langle \nabla |\nabla w|,\nabla w\rangle}{|\nabla w|},\\
    &=|\nabla w|^2-H|\nabla w|,
\end{align*}
and then we conclude that
\begin{align*}
    U'(t)&\leq 8\pi-U(t)+W(t)-\int_{\Sigma_t}\frac{|\nabla|\nabla w||^2}{|\nabla w|^2},
\end{align*}
as needed. \end{proof}

We will also need the following fact, which follows from \cite{agostiniani2024green}.

\begin{prop}
    The function $U$ is absolutely continuous. 
\end{prop}

\begin{proof}
In \cite{agostiniani2024green}, Agostiniani-Mazzieri-Oronzio proved
\[
\int_{\{w=s\}} H\vert \grad w\vert\, da = 4\pi + \int_{\{w=s\}} \vert \grad w\vert^2\, da - e^{-s}F(e^s),
\]
where $F$ is an absolutely continuous function. Here the reader should note that the notation used in \cite{agostiniani2024green} is different from ours: the function $u$ in our notation corresponds to the function $1-u$ in \cite{agostiniani2024green}.
Hence to prove $U$ is absolutely continuous, it suffices to show that 
\begin{align}\label{eq:absolutecont}
\int_{\{w=s\}} \vert \grad w\vert^2\, da
\end{align}
is absolutely continuous. For every regular value $s$, we have
\begin{align*}
\int_{\{w=s\}} \vert \grad w\vert^2\, da - \int_{\{w=0\}} \vert \grad w\vert^2\, da &= \int_{\{0<w<s\}} \div(\vert \grad w\vert \grad w)\, dv\\
&= \int_{\{0<w<s\}} \la \grad \vert \grad w\vert, \grad w\ra + \vert \grad w\vert \lap w\, dv.
\end{align*}
Since 
\[
\vert \grad w\vert \lap w = \vert \grad w\vert^3, \quad \big\vert \la \grad \vert \grad w\vert,\grad w\ra\big\vert  = \frac{1}{2\vert \grad w\vert}\big\vert \la \grad \vert \grad w\vert^2, \grad w\ra \big\vert \le \big\vert \grad \vert \grad w\vert^2\big\vert,
\]
the integrand over the bulk region is locally bounded. This implies that 
\[
\int_{\{w=s\}} \vert \grad w\vert^2 \, da
\]
is absolutely continuous (c.f. \cite{colding2021singularities} Lemma 3.14). By \eqref{eq:absolutecont}, we know $U$ is absolutely continuous. 
\end{proof}

Finally, we recall three consequences of Miao's monotonicity formulas \cite{miao2023mass}. 

\begin{prop}
    For each regular value $t$ of $w$, we have 
    \begin{itemize}
        \item[(i)] $3W(t) \le U(t) + 4\pi$,
        \item[(ii)] $U(t)\ge 8\pi$,
        \item[(iii)] $W(t) \ge 4\pi$.
    \end{itemize}
    In all three cases, equality holds when $\Omega$ is a ball. 
\end{prop}

\begin{proof}
Again we caution the reader that the notation in \cite{miao2023mass} is different from ours: the function $u$ in our notation corresponds to the function $1-u$ in \cite{agostiniani2020monotonicity}. Inequality (i) follows from equation (3.1) in \cite{miao2023mass}. Inequalities (ii) and (iii) follow from equations (3.15) and (3.16) in \cite{miao2023mass} since Euclidean space has zero ADM mass.
\end{proof}

Everything is now set up to estimate the Yamabe quotient of $f$. 



\begin{prop}
\label{f-quotient}
    We have 
    \[
    \frac{\int_{\mathbb R^3\setminus \Omega}|\nabla f|^2dv-\frac{1}{4}\int_{\partial\Omega}H f^2da}{\left(\int_{\mathbb R^3\setminus \Omega}f^6 dv\right)^{1/3}}\leq Q^*(B).
    \]
\end{prop}

\begin{proof} 
Note that 
\[
\grad f = s'(w) \vert \grad w\vert^{1/2} \grad w + \frac{1}{2} s(w) \frac{1}{\vert \grad w\vert^{1/2}}\grad \vert \grad w\vert
\]
almost everywhere. 
Therefore, with the above dominating ODEs and the co-area formula, we compute
\begin{align*}
    \int_{\mathbb R^3\setminus \Omega}|\nabla f|^2\, dv &=\int_0^{\infty}s'(t)^2W(t)+s(t)s'(t)(W(t)-U(t)) +\frac{1}{4}s(t)^2\int_{\Sigma_t}\frac{|\nabla|\nabla w||^2}{|\nabla w|^2}\, da\, dt \\
    &\le \int_0^{\infty}2\pi s(t)^2+(s'(t)+\frac{1}{2}s(t))^2W(t)-s(t)(s'(t)+\frac{1}{4}s(t))U(t)-\frac{1}{4}s(t)^2U'(t)\, dt. 
    \end{align*}
Since $U$ is absolutely continuous, we can now integrate by parts to get 
\begin{align*} 
    \int_{\mathbb R^3\setminus \Omega}|\nabla f|^2 &=\frac{1}{4}s(0)^2U(0)+\int_0^{\infty}2\pi s(t)^2+(s'(t)+\frac{1}{2}s(t))^2W(t)-\frac{1}{2}s(t)s'(t)U(t)-\frac{1}{4}s(t)^2U(t)\, dt \\
    &=\frac{1}{4}s(0)^2U(0)+\int_0^{\infty}2\pi s(t)^2+(s'(t)+\frac{1}{2}s(t))\left[(s'(t)+\frac{1}{2}s(t))W(t)-\frac{1}{2}s(t)U(t)\, dt\right].
\end{align*}
Recall Miao's results \cite{miao2023mass}:
\[
3W(t)\leq U(t)+4\pi,\quad U(t)\geq 8\pi. 
\]
Both inequalities achieve equality when $\Omega$ is the unit ball. 
We compute that
\begin{align}
    (s'(t)+\frac{1}{2}s(t))W(t)-\frac{1}{2}s(t)U(t)&=\frac{e^{t/2}}{(1+e^{2t})^{3/2}}W(t)-\frac{e^{t/2}}{2(1+e^{2t})^{1/2}}U(t)\\
    &\leq \frac{e^{t/2}}{3(1+e^{2t})^{3/2}}U(t)-\frac{e^{t/2}}{2(1+e^{2t})^{1/2}}U(t)+\frac{4\pi e^{t/2}}{3(1+e^{2t})^{3/2}}\\
    &=\frac{e^{t/2}}{(1+e^{2t})^{1/2}}\cdot\frac{-(1+3e^{2t})}{6(1+e^{2t})}U(t)+\frac{4\pi e^{t/2}}{3(1+e^{2t})^{3/2}} \label{l1}\\
    &\leq -\frac{4\pi e^{5/2t}}{(1+e^{2t})^{3/2}}.
    \end{align}
In the last step, we used Miao's inequality \cite{miao2023mass} and the fact that $-3e^{2t}-1\leq 0$ . With the above estimates, and the fact that 
\[
\frac{1}{4}s(0)^2 U(0) = \frac 1 4 \int_{\bd \Omega} Hf^2\, da,
\]
we bound $\int_{\mathbb R^3\setminus\Omega}|\nabla f|^2\, dv-\frac{1}{4}\int_{\partial\Omega}Hf^2\, da$ from above by $\int_0^{\infty}M(t)\,dt$, where
\[
M(t):=\frac{2\pi e^{t}(1+e^{4t})}{(1+e^{2t})^3}, \quad \text{and} \quad 
\int_0^{\infty}M(t)\,dt=\frac{3\pi^2}{8}.
\]
We also have that
\begin{align*}
    \int_{\mathbb R^3\setminus \Omega}f^6\, dv&=\int_0^{\infty}s(t)^6W(t)\, dt \geq \int_0^{\infty}4\pi  s(t)^6 \,dt = \frac{\pi^2}{8}.
\end{align*}
Here, we used Miao's \cite{miao2023mass} estimate $
W(t)\geq 4\pi$. 
Thus, we obtain
\begin{align*}
    \frac{\int_{\mathbb R^3\setminus \Omega}|\nabla f|^2dv-\frac{1}{4}\int_{\partial\Omega}Hf^2da}{\left(\int_{\mathbb R^3\setminus \Omega}f^6 dv\right)^{1/3}}\leq \frac{\int_0^{\infty}M(t)\, dt}{\left(4\pi\int_0^{\infty}s(t)^6\right)^{1/3}}= \frac{3}{4}\pi^{4/3} = Q^*(B),
\end{align*}
as desired. 
\end{proof} 

Finally, to be complete, we should check that $f$ is actually an admissible test function for the Yamabe quotient.

\begin{prop}
    The function $f$ belongs to $W^{1,2}(\R^3\setminus \Omega)$.
\end{prop}

\begin{proof}
    It follows from Proposition \ref{f-quotient} that $\grad f \in L^2(\R^3\setminus \Omega)$. Moreover, it is straightforward to check that $f \in L^2(\R^3\setminus \Omega)$ using the asymptotic expansion of the harmonic function $u$ near infinity. It remains to show that $\grad f$ is a weak derivative for $f$. 

    We will repeatedly use the fact that $\grad w\neq 0$ almost everywhere since $w = -\log u$ and $u$ is a non-constant harmonic function. Let $\phi \in C^\infty_c(\R^3 \setminus \cl \Omega,\R^3)$ be a smooth vector field with compact support.  For each number $\eps > 0$, define the function 
    \[
    f_\eps = s(w) \left( \vert \grad w\vert^2 + \eps^2\right)^{1/4}. 
    \]
    Then $f_\eps$ is smooth and so 
    \[
    \int_{\R^3\setminus \Omega} f_\eps \div(\phi)\, dv = - \int_{\R^3\setminus \Omega} \la \grad f_\eps,\phi \ra \,dv. 
    \]
    Note that $f_\eps \to f$ almost everywhere. 
    Therefore, we have 
    \[
    \int_{\R^3\setminus \Omega} f_\eps \div(\phi)\,dv \to \int_{\R^3\setminus \Omega} f\div(\phi)\, dv 
    \]
    by dominated convergence since $\phi$ has compact support. 

Next, observe that 
    \[
    \left\vert \int_{\R^3\setminus \Omega}  \la \grad f_\eps - \grad f, \phi\ra \,dv \right\vert  \le \|\phi\|_{L^2(\R^3\setminus \Omega)} \| \grad f_\eps - \grad f\|_{L^2(\supp (\phi))}. 
    \]
    Therefore, to complete the proof, it suffices to show that $\grad f_\eps \to \grad f$ in $L^2(K)$ for each compact set $K\subset \R^3\setminus \cl \Omega$. Observe that 
    \begin{gather*}
        \grad f = s'(w) \vert \grad w\vert^{1/2} \grad w + \frac{1}{2} s(w) \frac{1}{\vert \grad w\vert^{1/2}}\grad \vert \grad w\vert\\
        \grad f_\eps = s'(w) \left(\vert \grad w\vert^2 + \eps^2\right)^{1/4} \grad w + \frac{1}{2} s(w) \frac{\vert \grad w\vert}{\left(\vert \grad w\vert^2 + \eps^2\right)^{3/4}} \grad \vert \grad w\vert,
    \end{gather*}
    at all points where $\grad w \neq 0$. In particular, since $\grad w\neq 0$ on a set of full measure, it follows that $\grad f_\eps \to \grad f$ almost everywhere. Next, note that 
    \begin{align*}
        \vert \grad f_\eps\vert^2 &= s'(w)^2 \left(\vert \grad w\vert^2 + \eps^2\right)^{1/2} \vert \grad w\vert^2 + s(w)s'(w) \frac{\vert \grad w\vert}{\left(\vert \grad w\vert^2 + \eps^2\right)^{1/2}} \la \grad \vert \grad w\vert, \grad w\ra \\
        &\qquad \qquad + \frac 1 4 s(w)^2 \frac{\vert \grad w\vert^2}{\left(\vert \grad w\vert^2 + \eps^2\right)^{3/2}} \vert \grad \vert \grad w\vert\vert^2    
    \end{align*}
    almost everywhere. We have 
    \begin{align*}
        &\int_K s'(w)^2 \left(\vert \grad w\vert^2 + \eps^2\right)^{1/2} \vert \grad w\vert^2 + s(w)s'(w) \frac{\vert \grad w\vert}{\left(\vert \grad w\vert^2 + \eps^2\right)^{1/2}} \la \grad \vert \grad w\vert, \grad w\ra \, dv\\
        &\qquad \to \int_K s'(w)^2  \vert \grad w\vert^3 + s(w)s'(w)  \la \grad \vert \grad w\vert, \grad w\ra\, dv
    \end{align*}
    by dominated convergence since $K$ is compact. Finally, observe that 
    \[
    \int_K \frac 1 4 s(w)^2 \frac{\vert \grad w\vert^2}{\left(\vert \grad w\vert^2 + \eps^2\right)^{3/2}} \vert \grad \vert \grad w\vert\vert^2 \, dv \to \int_K \frac{1}{4}s(w)^2 \frac{\vert\grad \vert \grad w\vert\vert^2}{\vert \grad w\vert}\, dv
    \]
    by monotone convergence. It follows that 
    \[
    \int_K \vert \grad f_\eps\vert^2 \, dv \to \int_K \vert \grad f\vert^2 \, dv.
    \]
    Thus $\grad f_\eps,\grad f\in L^2(K)$ and $\grad f_\eps \to \grad f$ almost everywhere and $\| \grad f_\eps\|_{L^2(K)} \to \|\grad f\|_{L^2(K)}$. It is well-known that this implies that $\grad f_\eps \to \grad f$ in $L^2(K)$, as needed. 
\end{proof}

\subsection{Stability} Next, we refine Proposition \ref{f-quotient} so as to get an estimate on the difference $Q^*(B) - Q^*(\Omega)$. To do this, we need to keep track of an extra term in Miao's monotonicity formula:

\begin{prop}\label{prop:Ptestimate}
    We have 
    \[
    U(t)\ge 8\pi+\frac{1}{2}e^t\int_t^{\infty}e^{-\tau}\int_{\Sigma_{\tau}}\|\mathring{A}\|^2da.
    \]
\end{prop}

\begin{proof}
    Suppose $t$ is a regular value of $w$. Then we can compute
    \begin{align*}
        U'(t)&=\int_{\Sigma_t}H^2+H\frac{\langle \nabla|\nabla w|,\nabla w\rangle}{|\nabla w|^2}-|\nabla w|\left(\Delta_{\Sigma_t}\frac{1}{|\nabla w|}+(\|A\|^2+Ric(\nu_t,\nu_t)\frac{1}{|\nabla w|})\right)\\
        &\leq 4\pi+U(t)-\int_{\Sigma_t}\frac{3}{4}H^2da-\int_{\Sigma_t}\frac{1}{2}\|\mathring{A}\|^2da\\
        &\leq -8\pi+U(t)-\int_{\Sigma_t}\frac{1}{2}\|\mathring{A}\|^2da.
    \end{align*}
    We used the Gauss-Bonnet formula and the Willmore energy estimate of closed connected surfaces in Euclidean space.
    Note that 
    \[
    \lim_{t\to\infty}U(t)=8\pi,
    \]
    so with standard ODE argument we conclude the results.
\end{proof}
\begin{rem}
    Since we have shown the absolute continuity of $U(t)$, we do not need to worry about the critical points.
\end{rem}

\begin{prop}
\label{level-set-estimate}
    We have an estimate 
    \[
    \int_0^\infty b(\tau) \int_{\Sigma_\tau} \| \mathring A\|^2\, da \, d\tau \le Q^*(B) - Q^*(\Omega). 
    \]
    Here $b$ is a smooth function with $b(\tau) > 0$ for $\tau > 0$.  
\end{prop}

\begin{proof}
We can argue as in Proposition \ref{f-quotient} until line (\ref{l1}). At this point, we then apply the stronger inequality in Proposition \ref{prop:Ptestimate} to get 
\begin{align}
\label{s1}
    (s'(t)+\frac{1}{2}s(t))W(t)-\frac{1}{2}s(t)U(t)&\le -\frac{4\pi e^{5/2t}}{(1+e^{2t})^{3/2}} -\frac{e^{3t/2}(1+3e^{2t})}{12(1+e^{2t})^{3/2}}\int_t^{\infty}e^{-\tau}
    \int_{\Sigma_{\tau}}\|\mathring{A}\|^2da\, d\tau.
\end{align}
We can then give an explicit formula of $b(\tau)$ using Fubini's theorem. Indeed, we have
\begin{align*}
    &\int_0^{\infty}\frac{e^{2t}(1+3e^{2t})}{12(1+e^{2t})^{3}}dt\int_{t}^{\infty}e^{-\tau}\left(\int_{\Sigma_{\tau}}\|\mathring{A}\|^2da\right)d\tau\\
    &\qquad = \int_{0}^{\infty}e^{-\tau}\left(\int_{\Sigma_{\tau}}\|\mathring{A}\|^2da\right)d\tau\int_{0}^{\tau}\frac{e^{2t}(1+3e^{2t})}{12(1+e^{2t})^3}dt\\
    &\qquad = \int_0^{\infty}e^{-\tau}\left[\frac{(3e^{2\tau}+1)^2}{96(e^{2\tau}+1)^2}-\frac{1}{24}\right]\int_{\Sigma_\tau}\|\mathring{A}\|^2da.
\end{align*}
After normalization by the dominator, we obtain that
\[
b(\tau)= e^{-\tau}\left[\frac{(3e^{2\tau}+1)^2}{12\pi^2(e^{2\tau}+1)^2}-\frac{1}{3\pi^2}\right],
\]
is as required.
\end{proof}






\subsection{The Sobolev Quotient} To conclude this section, we show that a similar estimates hold for the Sobolev quotient. We should choose another test function $f_2(w)=s_2(w)|\nabla w|^{1/2}$. Here $s_2(w)=p(e^w)=e^{-w/2}$, and $p(r)$ is the conformal factor from the standard stereographical projection map between the standard round sphere and $\mathbb R^3$. Computation for the numerator is exactly the same if we plugging in $f_2$ instead of $f$. More precisely, we obtain
\begin{align*}
    \int_{\mathbb R^3\setminus\Omega}|\nabla f_2|^2\leq\frac{1}{4}s_2(0)U(0)+\int_0^{\infty}2\pi s_2(t)^2+(s'_2(t)+\frac{1}{2}s_2(t))\left[(s_2'(t)+\frac{1}{2}s_2(t))W(t)-\frac{1}{2}s_2(t)U(t)\right].
\end{align*}
Note that
\[
s_2'(t)+\frac{1}{2}s_2(t)=0.
\]
We easily obtain that
\[
\int_{\mathbb R^3\setminus\Omega}|\nabla f_2|^2-\frac{1}{4}\int_{\partial\Omega}Hf_2^2\leq \int_0^{\infty}2\pi e^{-t}dt=2\pi
\]
As for the denominator, we have
\begin{align*}
    \int_{\partial \Omega}f^4&=\int_{\partial\Omega}s_2^4|\nabla w|^2 =W(0).
\end{align*}
Now, we obtain a sharper estimate for the dominator.

\begin{lem}
    We have the following estimate for $K(0)$,
    \[
    W(0)\geq 4\pi+4\pi\int_0^{\infty}(1-e^{-\tau})e^{-\tau}\left[\int_{\Sigma_{\tau}}\|\mathring{A}\|^2da\right]d\tau
    \]
\end{lem}
\begin{proof}
    Suppose $t$ is a regular value of $w$. We have 
    \begin{align*}
        W'(t)=2W(t)-U(t).
    \end{align*}
    With standard ODE techniques, we obtain
    \begin{align*}
        W(0)&=\int_0^{\infty}e^{-2t}U(t)\\
            &\geq\int_0^{\infty}e^{-2t}\left[8\pi+4\pi e^t\int_t^{\infty}e^{-\tau}\int_{\Sigma_{\tau}}\|\mathring{A}\|^2\right]dt\\
            &=4\pi+4\pi\int_0^{\infty}e^{-t}dt\int_{t}^{\infty}e^{-\tau}\left[\int_{\Sigma_{\tau}}\|\mathring{A}\|^2 da\right]d\tau\\
            &=4\pi+4\pi\int_0^{\infty}(1-e^{-\tau})e^{-\tau}\left[\int_{\Sigma_{\tau}}\|\mathring{A}\|^2da\right]d\tau.
    \end{align*}
    We used Proposition \ref{prop:Ptestimate} and Fubini's theorem in the above computations.
\end{proof}

We define $
b_2(\tau)=(1-e^{-\tau})e^{-\tau}$. 
We can then obtain the following sharper estimate of the Sobolev Quotient.

\begin{prop} \label{sobolev-quotient-est} Assume that $\Omega$ is a smooth, bounded domain in $\R^3$ with connected boundary. If $Q^{\bd, *}(\Omega)$ is sufficiently close to $Q^{\bd,*}(B)$ then we have an estimate 
\[
Q^{\partial,*}(\Omega)\leq \frac{2^{\frac{1}{2}}\pi^{\frac{3}{4}}}{\left(1+\int_0^{\infty}b_2(\tau)\int_{\Sigma_{\tau}}\|\mathring{A}\|^2da\right)^{\frac{1}{4}}}\leq Q^{\partial,*}(B)\left(1-\frac{1}{8}\int_0^{\infty}\int_{\Sigma_{\tau}}\|\mathring{A}\|^2da\right).
\]
\end{prop}

\section{Main Results}

\label{Section:Stability}

In this section, we prove the main theorems.  Given a bounded domain $D\subset \R^3$ with smooth boundary, we let 
\[
\cp(D) = \inf \left\{\int_{\R^3} \vert \grad u\vert^2\, dv: u\in W^{1,2}(\R^3),\ u\ge 1 \text{ on } D\right\}
\]
denote the capacity of $D$. First we need an elementary result, characterizing subsets of a ball with nearly the same capacity as the ball.

\begin{prop}
\label{capacity-est}
    There is a constant $c > 0$ so that if $\Omega\subset B(x,r)$ is a smooth domain with 
    \[
    \frac{\cp(\Omega)}{\cp(B(x,r))}\ge 1 - c\eps^3 
    \]
    then $\bd B(x,r)$ is contained in the $\eps r$-neighborhood of $\bd \Omega$. 
\end{prop}

\begin{proof}
    The statement is scale invariant and translation invariant so we may as well assume that $B(x,r) = B(0,1)$. We will prove the contrapositive. Fix some $\eps > 0$. Consider $\Omega \subset B(0,1)$ and assume there is a point $p\in \bd B(0,1)$ such that $\bd \Omega \cap B(p,\eps) = \emptyset$. After rotation, we may assume that $p = (0,0,-1)$. Let $B_1 = B(0,1)$ and let $B_2 = B(p,\eps)$. Then $\Omega \subset \Omega_\eps$ for $\Omega_\eps = B_1\setminus B_2$. 
    Consequently, it suffices to show that $\cp(\Omega_\eps) \le 1 - c\eps^3$. 

    It is enough to find a good test function. We construct a test function $f$ as follows. First, let $f = 1/r$ in the complement of $B_1 \cup B_2$. Write $\bd B_2 = \Sigma \cup \Gamma$ where $\Sigma = B_1 \cap \bd B_2$ and $\Gamma = (\bd B_2)\setminus B_1$. In $B_2$, we let $f$ be the radial function which is equal to $1/r$ on $\Gamma$, equal to $1$ on $\Sigma$, and equal to $1 - \frac{\eps}{\pi}$ at $p$. To prove the estimate, it suffices to show that 
    \begin{equation}
    \label{1}
    \int_{B_2\setminus B_1} \vert \grad r^{-1}\vert^2 \, dv - \int_{B_2} \vert \grad f\vert^2\, dv \ge c \eps^3. 
    \end{equation}
    The first integral can be estimated by 
    \begin{equation}
    \label{2}
    \int_{B_2\setminus B_1} \vert \grad r^{-1}\vert^2 \, dv \ge \left(\frac{2\pi}{3}\eps^3\right) \left(\frac{1}{(1+\eps)^2}\right) = \frac{2\pi}{3}\eps^3 - O(\eps^4).
    \end{equation}
    It remains to estimate the second integral. 

    Introduce a spherical coordinate system $(\rho,\theta,\phi)$ centered at $p$. Thus 
\[
\begin{cases}
    x = \rho \cos \theta \sin \phi,\\
    y = \rho \sin \theta \sin \phi,\\
    1+z = \rho \cos \phi.
\end{cases}
\]
In these coordinates, $\Sigma$ is the set $\rho = \eps$, $\phi \le \phi_0$ and $\Gamma$ is the set $\rho = \eps$, $\phi \ge \phi_0$ where $\phi_0 = \arccos(\frac{\eps}{2})$. Define
\[
g(\phi) = \begin{cases}
    1, &\phi \le \phi_0\\
    \frac{1}{\sqrt{1+\eps^2-2\eps \cos \phi}}, &\phi \ge \phi_0. 
\end{cases}
\]
Then in $B_2$ the test function $f$ is given by 
\[
f(\rho,\theta,\phi) = \frac{\left(1-\frac{\eps}{\pi}\right)(\eps - \rho) + \rho g(\phi)}{\eps}.
\]
It follows that 
\[
\vert \grad f\vert^2 = \frac{(g(\phi)- 1 + \frac{\eps}{\pi})^2 + (g'(\phi))^2}{\eps^2}. 
\]
Therefore we have 
\begin{align*}
\int_{B_2} \vert \grad f\vert^2 \, dv &= 2\pi \int_{\phi = 0}^{\pi} \int_{\rho = 0}^\eps \left(\frac{(g(\phi)- 1 + \frac{\eps}{\pi})^2 + (g'(\phi))^2}{\eps^2}\right) \rho^2 \sin \phi \, d\rho \, d\phi\\
&= \frac{2\pi \eps}{3} \int_{\phi = 0}^\pi \left[\left(g(\phi)- 1 + \frac{\eps}{\pi}\right)^2 + (g'(\phi))^2\right] \sin \phi\, d\phi. 
\end{align*}
We split this integral into two pieces according to the piecewise definition of $g$. We have 
\begin{align*}
\int_0^{\phi_0} \left[\left(g(\phi)- 1 + \frac{\eps}{\pi}\right)^2 + (g'(\phi))^2\right] \sin \phi\, d\phi = \int_0^{\phi_0} \frac{\eps^2}{\pi^2} \sin\phi\, d\phi = \frac{\eps^2}{\pi^2} \left(1 - \frac{\eps}{2}\right). 
\end{align*}
For $\phi \ge \phi_0$ we can Taylor expand in $\eps$ to get 
\[
g(\phi) = 1 + \eps \cos(\phi) + O(\eps^2),\quad  g'(\phi) = -\eps \sin(\phi) + O(\eps^2). 
\]
Therefore we have 
\begin{align*}
    \int_{\phi_0}^{\pi} \left[\left(g(\phi)- 1 + \frac{\eps}{\pi}\right)^2 + (g'(\phi))^2\right] \sin \phi\, d\phi &= \int_{\phi_0}^\pi \left(\eps \cos \phi + \frac{\eps}{\pi}\right)^2 \sin \phi + \eps^2 \sin^3\phi \, d\phi + O(\eps^3)\\
    &= \left(1 - \frac{1}{\pi} + \frac{1}{\pi^2}\right)\eps^2 + O(\eps^3). 
\end{align*}
Hence we obtain an estimate 
\begin{align*}
    \int_{B_2} \vert \grad f\vert^2 \, dv &= \frac{2\pi \eps}{3}\left[\frac{\eps^2}{\pi^2} + \eps^2 - \frac{\eps^2}{\pi} + \frac{\eps^2}{\pi^2} + O(\eps^3) \right]\\
    &= \frac{2\pi}{3}\eps^3 \left[1 +\frac{2}{\pi^2} - \frac{1}{\pi} \right] + O(\eps^4).
\end{align*}
Since $\frac{2}{\pi^2}-\frac{1}{\pi} < 0$, combining this with (\ref{2}) gives (\ref{1}), as needed.
\end{proof}


\begin{prop}
    For every $\eta > 0$ there is a $\delta > 0$ so that if $Q^*(B) - Q^*(\Omega) < \delta$ then there is a number $1-\eta < s < 1$ such that 
    \[
    \int_{\{u=s\}} \|\mathring A\|^2 \, da \le \eta. 
    \]
\end{prop}

\begin{proof}
Fix some $\eta > 0$. Let $c(\eta) = -\log\left(1-\frac \eta 2\right)$ and $d(\eta) = -\log(1-\eta)$. Also note that 
\[
a(\eta) = \min\{ b(\tau): c(\eta) \le \tau \le d(\eta)\} > 0. 
\]
Therefore, according to Proposition \ref{level-set-estimate}, we have 
\[
a(\eta) \int_{c(\eta)}^{d(\eta)} \int_{\{w=\tau\}} \|\mathring A\|^2 \, da \, d\tau \le Q^*(B)-Q^*(\Omega).
\]
Hence there is some $t \in [c(\eta),d(\eta)]$ for which 
\[
\int_{\{w=t\}} \|\mathring A\|^2\, da \le \frac{Q^*(B)-Q^*(\Omega)}{a(\eta)(d(\eta)-c(\eta))}. 
\]
Note that $\{w=t\} = \{u=s\}$ where $s = e^{-t} \in [\frac{\eta}{2},\eta]$. Therefore, the result follows provided we select $\delta = \eta a(\eta) (d(\eta)-c(\eta)) > 0$. 
\end{proof}

Now we can prove a preliminary stability result for exterior domains. In the following, $\Omega$ always denotes a smooth, bounded domain in $\R^3$ with connected boundary. 

\begin{prop}
\label{prelim-stability1}
For every $\eps > 0$ there is a $\delta > 0$ such that if $Q^*(B) - Q^*(\Omega) < \delta$ then there is a ball $B(x,r)$ such that $\Omega \subset B(x,r)$ and $\bd B(x,r)$ is contained in the $\eps r$-neighborhood of $\bd \Omega$. 
\end{prop}

\begin{proof} Fix some $\eps > 0$. Then choose $\eta = \eta(\eps) > 0$ to be specified later. 
Let $u$ be the capacitary potential for $\Omega$. Let $\Gamma_t = \{u=t\}$ and note that $\bd \Omega = \Gamma_1$. By Proposition \ref{level-set-estimate}, we know that if $\delta$ is small enough depending on $\eta$, then there exists $1 - \eta  < s < 1$ such that 
\[
\int_{\Gamma_s} \| \mathring A\|^2\, da \le \eta.
\]
We denote by $C$ an absolute constant which is allowed to change from line to line. By the De Lellis-M\"uller almost-umbilic estimate \cite{de2005optimal}, there is a point $x\in \R^3$ and a radius $R > 0$ and a conformal paramaterization $\psi\colon S^2\to \Gamma_s$ satisfying 
 \[
 \|\psi - x - R\text{Id}\|_{W^{2,2}(S^2)} \le C R \|\mathring A\|_{L^2(\Gamma_s)}^2 \le C R \eta. 
 \]
Let $\Omega_s$ denote the region enclosed by $\Gamma_s$. Then by the continuous embedding of $W^{2,2}$ into $L^\infty$ we have $B(x,R(1-C\eta)) \subset  \Omega_s \subset B(x,R(1+C\eta))$.  

Choose a number $r < R(1+C\eta)$ so that $\Omega \subset \Omega_s \subset B(x,r)$. It remains to prove that $\bd B(x,r)$ is contained in the $\eps r$-neighborhood of $\bd \Omega$.  Since $\Sigma_s$ is a level set of the capacitary potential, we have 
\[
\frac{\cp( \Omega)}{\cp(\Omega_s)} = \frac{1}{s} \ge 1-C\eta. 
\]
It follows that 
\begin{align*}
\frac{\cp(\Omega)}{\cp(B(x,r))} &= \left(\frac{\cp{(\Omega)}}{\cp (\Omega_s)}\right)\left(\frac{\cp(\Omega_s)}{\cp(B(x,r))} \right)\\
&\ge (1-C\eta) \frac{\cp(B(x,R(1-C\eta)))}{\cp(B(x,R(1+C\eta)))} \\
&= (1-C\eta)\frac{4\pi R(1-C\eta)}{4\pi R(1+C\eta)} \ge 1-C\eta. 
\end{align*} 
Therefore, assuming $\eta = \eta(\eps)$ is chosen small enough, we can apply Proposition \ref{capacity-est} to conclude. 
\end{proof}

\begin{rem}
    In Proposition \ref{prelim-stability1}, it suffices to choose $\eta = O(\eps^3)$. Therefore, by the explicit formula for $b(\tau)$, it suffices to choose $\delta = \eta a(\eta)(d(\eta)-c(\eta)) = O(\eta^3) = O(\eps^9)$.  
\end{rem}

Now we want to apply this preliminary result to deduce stronger stability. 

\begin{defn}
Let $\Omega\subset \R^3$ be a smooth, bounded domain. Define $\mathcal B(\Omega)$ to be the set of all open balls $B$ such that $B \subset \Omega$. 
\end{defn}

It is obvious that $\mathcal B(\Omega)$ is non-empty. 

\begin{prop}
\label{max-ball}
There exists a ball $B\in \mathcal B(\Omega)$ with maximal radius. Moreover, every hemisphere of $\bd B$ intersects $\bd \Omega$. 
\end{prop}

\begin{proof}
Define $r = \sup\{\rho > 0: \text{there exists a ball $B$ of radius $\rho$ in $\mathcal B(\Omega)$}\}$. We have $r < \infty$ since $\Omega$ is bounded. We can choose a sequence of balls $B(x_i,r_i)\in \mathcal B(\Omega)$ with $r_i \to r$. After passing to a subsequence, we can suppose that $x_i\to x$. Then it is clear that $B = B(x,r) \in \mathcal B(\Omega)$. 

Now consider some unit vector $w\in S^2$. Consider the hemisphere
\[
S = \{y\in \bd B: \la y-x,w\ra \ge 0\}. 
\]
Suppose for contradiction that $S \cap \bd \Omega$ is empty. Then for some $\eps > 0$, the $\eps$-neighborhood of $S$ is contained in $\Omega$. Consequently, for sufficiently small $\delta > 0$, the closed ball $\cl B(x+\delta w, r)$ is contained in $\Omega$. This contradicts that $r$ was the maximal radius of a ball in $\mathcal B(\Omega)$ since now $B(x+\delta w, r  +\eta)$ is contained in $\Omega$ for sufficiently small $\eta > 0$. 
\end{proof}

\begin{defn}
Let $\mathcal A(\Omega)$ be the set of all balls $B = B(x,r)\in \mathcal B(\Omega)$ such that there exist points $y,z\in \bd B\cap \bd \Omega$ with $\angle yxz \ge \frac{\pi}{4}$. 
\end{defn}

Let $\mathcal M$ be the medial axis of $\Omega$. Observe that a point $x\in \Omega$ is the center of a ball $B\in \mathcal A(\Omega)$ if and only if $x\in \mathcal M$ and the separation angle at $x$ is at least $\frac{\pi}{4}$. 

\begin{prop}
There exists a ball $B\in \mathcal A(\Omega)$ with minimal radius. 
\end{prop} 

\begin{proof}
Proposition \ref{max-ball} implies that the set $\mathcal A(\Omega)$ is non-empty. Therefore 
\[
r = \inf\{\rho > 0: \text{there exists a ball $B$ of radius $\rho$ in $\mathcal A(\Omega)$}\}
\]
is well-defined. Moreover, since $\Omega$ is smooth, it is clear that $r > 0$. We can choose a sequence of balls $B(x_i,r_i)\in \mathcal A(\Omega)$ with $r_i\to r$. After passing to a subsequence, we can suppose that $x_i\to x$. Then it is clear that $B(x,r)\in \mathcal A(\Omega)$ is a ball with minimal radius. 
\end{proof}


Next we prove a structural result which shows that if $Q(\Omega)$ is close to $Q(B)$ then $\Omega$ must take a particular form near each ball in $\mathcal A(\Omega)$. 

\begin{prop}
\label{A-structure1}
For every $\eps > 0$ there is a $\delta > 0$ so that if $Q(B)-Q(\Omega) < \delta$ then the following property holds: whenever $B(x,r)\in \mathcal A(\Omega)$ it follows that $\bd B(x,r)$ is contained in the $\eps r$-neighborhood of $\bd \Omega$.  
\end{prop}

\begin{proof}
Fix a ball $B = B(x,r)\in \mathcal A(\Omega)$. Since the conclusion is translation invariant and scale invariant, we can assume without loss of generality that $B = B(0,1)$. 
Let $\phi(x) = x/\vert x\vert^2$ be the conformal inversion centered at the origin and let $\Omega^*$ be the interior of $\R^3\setminus [\phi(\Omega \setminus \{0\})]$. Then $Q^*(B) - Q^*(\Omega^*) = Q(B)-Q(\Omega)$.  Hence, by the previous results, if $\delta$ is small enough then there is a ball $B^* = B(x^*,r^*)$ such that $\Omega^*\subset B^*$ and $\bd B^*$ is contained in the $\frac{\eps}{100}r^*$-neighborhood of $\bd \Omega^*$.  

We claim that $1-\eps < r^* < 1+\eps$. Indeed, if $r^* \ge 1+\eps$, then since $\Omega^*\subset B(0,1)$, it follows that there is a point $z\in \bd B^*$ with $d(z,\bd \Omega^*) \ge d(z,\bd B(0,1)) \ge r^*-1$. This implies that 
\[
r^*-1 \le \frac{\eps}{100} r^* 
\]
and therefore that 
\[
r^* \le \frac{1}{1-\frac{\eps}{100}} < 1 + \eps
\]
which is a contradiction. 

On the other hand, suppose that $r^* < 1 -\eps$. 
Then since $\bd \Omega^*$ contains two points on the unit sphere making an angle of at least $\frac \pi 4$, it follows by elementary geometry that there must be a point $z\in \bd B^*$ which is far away from $\bd B(0,1)$.
\[
\includegraphics[width=3.5in]{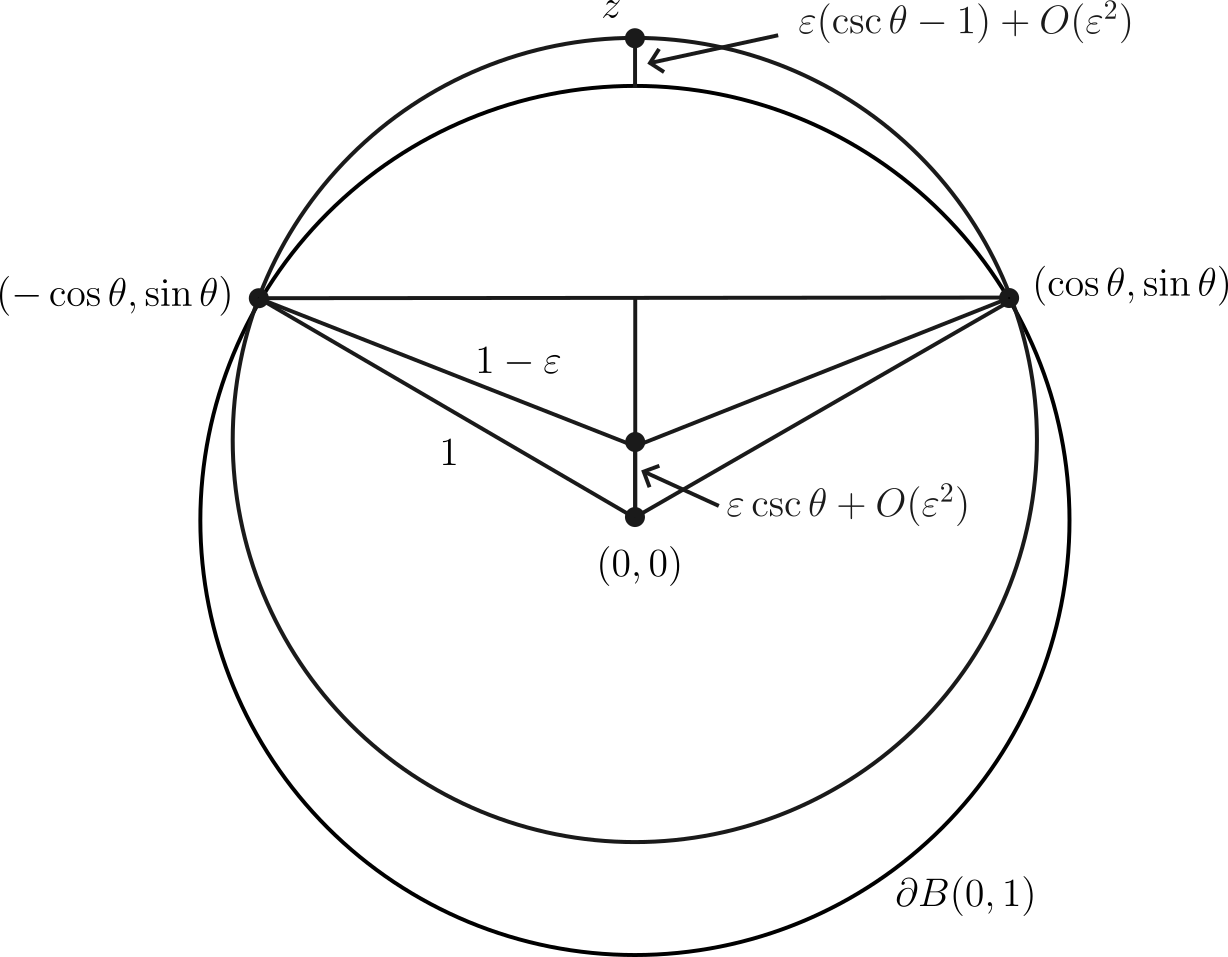}
\]
Indeed, in the above figure we have $0\le\theta \le \frac {3 \pi}{8}$ and therefore the point $z$ at the top of the figure is at distance at least $\eps(\csc \theta - 1) + O(\eps^2) \ge \frac{\eps}{10}$ from $\bd B(0,1)$. Thus we have  $d(z,\bd \Omega^*) \ge d(z,\bd B(0,1)) \ge \frac{\eps}{10}$.  This contradicts that $\bd B^*$ is contained in the $\frac{\eps}{100}r^*$-neighborhood of $\bd \Omega^*$. 
Therefore we have $1-\eps < r^*< 1+\eps$.  Similar considerations now show that  $d(x^*,0) \le 2\eps$. 
Therefore $B^*$ is very close to the unit ball. 

The conclusion now follows by undoing the conformal inversion. Indeed, consider some point $y\in \bd B(0,1)$. We know that there is a point $y^*\in \bd B^*$ with $d(y^*,y) \le 3\eps$. Moreover, there is a point $z^* \in \bd \Omega^*$ with $d(y^*,z^*) \le \frac{\eps}{100}r^* \le \eps$. The triangle inequality implies that $d(y,z^*) \le 4\eps$. In particular, we have $1 - 4\eps \le \vert z^*\vert \le 1+4\eps$. Let $z = \phi(z^*) \in \bd \Omega$ and note that $y = \phi(y)$. Hence, provided $\eps$ is small enough, we can estimate 
\begin{align*}
    \vert y - z\vert = \left\vert y - \frac{z^*}{\vert z^*\vert^2}\right\vert = \frac{\vert \vert z^*\vert^2 y - z^*\vert}{\vert z^*\vert^2} &\le  \frac{\vert z^*\vert^2 \vert y-z^*\vert + \vert \vert z^*\vert^2 z^*-z^*\vert}{\vert z^*\vert^2}\\
    &= \vert y - z^*\vert + \frac{\vert \vert z^*\vert^2 - 1\vert}{\vert z^*\vert}\\
    &\le 4\eps + \frac{(4\eps)^2}{1 - 4\eps} \le 5\eps.
\end{align*}
This proves the proposition. 
\end{proof}

Now we can proceed with the main results. At this point, the reader may wish to review the flow $\mathfrak C$ discussed in Section \ref{Section:Medial}.

\begin{rethm}
    For any $\eps > 0$ there is a $\delta > 0$ such that if $Q(B) - Q(\Omega) < \delta$ then there is a ball $B(x,r)$ such that $B(x,r) \subset \Omega \subset B(x,r(1+\eps))$.   
\end{rethm}

\begin{proof}
     Let $B\in \mathcal A(\Omega)$ be a ball of minimal radius. Since the statement is scale invariant and translation invariant, we may as well assume that $B = B(0,1)$. Then we need to show that $\Omega\subset B(0,1+\eps)$. 
     
     Applying Proposition \ref{A-structure1}, if $\delta$ is small enough, then $\bd B(0,1)$ is contained in the $\frac{\eps}{100}$-neighborhood of $\bd \Omega$. 
Assume for contradiction that there is a point $y\in \Omega$ with $\vert y\vert \ge 1+\eps$. To prove the proposition, we need to argue that  there is actually  a ball in $\mathcal A(\Omega)$ with radius smaller than 1.   
By Proposition \ref{angle}, it suffices to find a point $p\in \Omega$ with $\mathcal R(p) < 1$ and $\|\nabla (p)\|^2 \le \frac{1}{2}$. We distinguish two cases. 

Case 1: There is a point $x\in \Omega$ with $\vert x\vert \ge 1$ and $B(x,\frac{\eps}{2})\subset \Omega$. In this case, $\Omega$ must contain a neck. More precisely, let $P$ be the set of all continuous paths $\gamma\colon [0,1]\to \Omega$ connecting the origin to $x$. Define 
\[
r_0 = \sup\{r>0: \text{there exists } \gamma\in P \text{ with } B(\gamma(t),r)\subset \Omega \text{ for all } t\in[0,1]\}. 
\]
Under the assumptions of Case 1, it is clear that $r_0 < \frac{\eps}{4}$.  Let $\Omega_a = \{z\in \Omega: d(z,\bd \Omega)\le a\}$. We claim there is a point in $p\in \Omega_{\eps/2}$ with $\|\del(p)\|^2\le 1/2$. Suppose instead that $\|\del(p)\|^2\ge 1/2$ for all $p\in \Omega_{\eps/2}$.  Let $\tau = d(\Omega_{\eps/4},\bd \Omega_{\eps/2}) > 0$. 
Choose a number $\eta > 0$ much smaller than $\tau$. Then choose a number $r_0 - \eta < r < r_0$ and a path $\gamma\in P$ with $B(\gamma(t),r)\subset \Omega$ for all $t$. Consider the new path $\tilde \gamma(t) = \mathfrak C(\tau,\gamma(t))$.  This is still continuous. Consider any time $s \in [0,1]$ with $\mathcal R(\gamma(s)) \le r_0 + \eta$. Then $\gamma(s) \in \Omega_{\eps/4}$ and so $\mathfrak C(\sigma,\gamma(s))\in \Omega_{\eps/2}$ for all $\sigma\in[0,\tau]$. Therefore we have 
\[
\mathcal R(\tilde \gamma(s)) \ge \mathcal R(\gamma(s)) + \int_0^\tau \|\mathfrak C(\sigma,\gamma(s))\|^2 \, d\sigma \ge \mathcal R(\gamma(s)) + \frac{\tau}{2} > r_0 + \eta.
\]
It follows that we have $\mathcal R(\tilde \gamma(t)) > r_0+ \eta$ for all $t\in [0,1]$. Finally, we can append to $\tilde \gamma$ paths from $\tilde \gamma(0)$ to the origin and from $\tilde \gamma(1)$ to $x$ and then re-parameterize to obtain a path $\hat \gamma$ in $P$ with $B(\hat\gamma(t), r_0+\eta)\subset \Omega$ for all $t$. This contradicts the definition of $r_0$. Hence there is a point $p\in \Omega_{\eps/4}$ with $\|\del(p)\|^2 \le 1/2$. This means there is a ball in $\mathcal A(\Omega)$ centered at $p$ with radius less than 1, as needed. 

Case 2: For every $x\in \Omega$ with $\vert x\vert\ge 1$ we have $d(x,\bd \Omega) \le \frac{\eps}{4}$.  Recall we are assuming that there is a point $y\in \Omega$ with $\vert y\vert \ge 1 + \eps$. Starting from this point $y$, we can follow the flow and we have 
\begin{gather*}
    \vert \mathfrak C(t,y) \vert \ge 1 + \eps - t.
\end{gather*}
Therefore the flow starting from $y$ remains outside $B(0,1)$ for time at least $\eps$. Moreover, we have 
\[
\mathcal R(\mathfrak C(t,y)) = \mathcal R(y) + \int_0^t \|\nabla(\mathfrak C(\tau,y))\|^2\, d\tau. 
\]
Since $\mathcal R(\mathfrak C(t,y)) \le \frac \eps 4$ for $t\in [0,\eps]$, if follows that for some $\tau\in [0,\eps]$ we have $\|\nabla(\mathfrak C(\tau,y))\|^2 \le \frac{1}{2}$. This point $p=\mathfrak C(\tau,y)$ is now the center of a  ball in $\mathcal A(\Omega)$ with radius less than 1. This completes the proof.  
\end{proof}

Actually, inspecting the above proof, we've really shown that any point $x\in \Omega \setminus B(0,1)$ can be joined to $B(0,1)$ by a path of length at most $\eps$ which stays inside $\Omega$. 
This implies  that $\Omega$ is Gromov-Hausdorff close to $B(0,1)$ as a metric space with its induced length metric. 

\begin{defn}
    Let $(X,d_X)$ and $(Y,d_Y)$ be two compact metric spaces. We say a (not necessarily continuous) map $f\colon X\to Y$ is an $\eps$-Gromov-Haussdorff approximation if 
    \begin{itemize}
        \item[(i)] the $\eps$ neighborhood of $f(X)$ in $Y$ is all of $Y$; and
        \item[(ii)] for all $x_1,x_2\in X$ we have $\vert d_Y(f(x_1),f(x_2))-d_X(x_1,x_2)\vert < \eps$.
    \end{itemize}
\end{defn}

It is well-known that if there exist $\eps$-Gromov Haussdorff approximations $f\colon X\to Y$ and $h\colon Y\to X$ then the Gromov-Hausdorff distance between $X$ and $Y$ is at most $\eps$. 

\begin{rethm}
    For every $\eps > 0$ there is a $\delta > 0$ so that if $Q(B) - Q(\Omega) < \delta$ then, after scaling and translation, the Gromov-Hausdorff distance between the unit ball and $\Omega$ with its induced length metric is at most $\eps$.  
\end{rethm}

\begin{proof}
    As in the previous proof, after translation and scaling, we can assume that $B(0,1)$ is a ball of minimal radius in $\mathcal A(\Omega)$. Then, assuming $\delta$ is small enough, we have $B(0,1) \subset \Omega \subset B(0,1+\eps)$. Case 2 of the previous proof actually gives the following stronger property: for every $x\in \Omega \setminus B(0,1)$ there is a rectifiable path $\gamma\colon [0,1]\to \Omega$ with $\gamma(0) = x$ and $\gamma(1)\in B(0,1)$ such that $\gamma$ has length at most $\eps$.  

    Let $d$ denote the induced length metric on $\Omega$. Note that the induced length metric on the unit ball $B$ agrees with the Euclidean distance $d_{\text{euc}}$. We want to show there is a Gromov-Hausdorff approximation between $(\Omega,d)$ and $(B,d_{\text{euc}})$. We are going to construct $\eps$-Gromov-Hausdorff approximations between $B$ and $\Omega$.
    
    First consider the natural inclusion $\iota\colon B\to \Omega$. Note that $d(\iota(x),\iota(y)) = d_{\text{euc}}(x,y)$. Moreover, by the property discussed above, we know that the $\eps$-neighborhood of $\iota(B)$ in $(\Omega,d)$ is all of $\Omega$. 
    On the other hand, define a map $f\colon \Omega\to B$ as follows. First, if $x\in B$ then let $f(x) = x$. Second, if $x\in \Omega \setminus B(0,1)$ then let $f(x)$ be a point of $B$ such that there is a path from $x$ to $f(x)$ in $\Omega$ with length at most $\eps$. Then the image $f(\Omega)$ is all of $B(0,1)$. Finally, note that for all $x,y\in \Omega$ we have 
    \begin{align*}
        \vert d_{\text{euc}}(f(x),f(y)) - d(x,y)\vert &= \vert d(f(x),f(y)) - d(x,y)\vert \le 2\eps.
    \end{align*}
    Thus the Gromov-Hausdorff distance between $(\Omega,d)$ and $(B,d_{\text{euc}})$ is at most $2\eps$.  
\end{proof}

Finally we prove Theorem \ref{main3}.

\begin{rethm}
    There is a $\delta > 0$ so that if $\Omega\subset \R^3$ is a smooth, bounded domain with connected boundary and $Q(B) - Q(\Omega) < \delta$ then $\Omega$ is diffeomorphic to a ball.
\end{rethm}

\begin{proof}
    Choose $\eps = \frac 1 {10}$. Again, after translation and scaling, we can assume that $B(0,1)$ is a ball of minimal radius in $\mathcal A(\Omega)$. Then we have $B(0,1)\subset \Omega \subset B(0,1+\eps)$. By Case 1 of the proof of  Theorem \ref{main1}, we know that for every $x\in \Omega \setminus B(0,1)$ we have $\mathcal R(x) \le \frac{\eps}{2}$. Note also that we must have $\|\del(x)\|^2 \ge \frac 1 2$ for all $x\in \Omega\setminus B(0,\frac{1}{2})$. It follows that for all $x\in \Omega \setminus B(0,3/4)$ we have $\mathfrak C(1/4,x) \in B(0,1)$. Indeed, note that $\| \del \mathfrak C(t,x)\|^2 \ge \frac 1 2$ for $t\in [0,1/4]$ and therefore $\mathcal R(\mathfrak C(1/4,x)) \ge 1/8$ which forces $\mathfrak C(1/4,x)\in B(0,1)$.  We also have $\mathfrak C(1/4,x)\in B(0,1)$ for all  $x\in B(0,3/4)$. Therefore, there is a homotopy $H\colon \Omega\times [0,1]\to \Omega$ with $H(1,\Omega)\subset B(0,1)$. This implies that $\Omega$ is contractible (as a subset of itself). Therefore $\bd \Omega$ must be a sphere by the half lives, half dies theorem. It now follows that $\Omega$ is diffeomorphic to a ball since every smoothly embedded $S^2$ in $\R^3$ bounds a 3-ball. 
\end{proof}

To conclude this section, we note that Theorems \ref{main1}-\ref{main3} all have counterparts for the Sobolev quotient $Q(\Omega,\bd \Omega)$. The proofs are identical, once one replaces Proposition \ref{level-set-estimate} with Proposition \ref{sobolev-quotient-est}. Hence we have the following: 

\begin{theorem}
For every $\eps > 0$ there is a $\delta > 0$ so that the following is true: whenever $\Omega\subset \R^3$ is a smooth, bounded domain with connected boundary and $Q(B,\bd B) - Q(\Omega,\bd \Omega) < \delta$ then there exists a ball $B(x,r)$ such that $B(x,r) \subset \Omega \subset B(x,r(1+\eps))$.
\end{theorem} 

\begin{theorem} 
For every $\eps > 0$ there is a $\delta > 0$ so that the following is true: whenever $\Omega\subset \R^3$ is a smooth, bounded domain with connected boundary and $Q(B,\bd B) - Q(\Omega,\bd \Omega) < \delta$ then, after suitable scaling of $\Omega$, the Gromov-Hausdorff distance between the unit ball and $\Omega$ equipped with its induced length metric is at most $\eps$. 
\end{theorem}

\begin{theorem}
    There is a $\delta > 0$ so that if $\Omega\subset \R^3$ is a smooth, bounded domain with connected boundary and $Q(B,\bd B) - Q(\Omega,\bd \Omega) < \delta$ then $\Omega$ is diffeomorphic to a ball. 
\end{theorem}

\section{Comparison with the Coefficient of Quasi-conformality}

\label{Section:Comparison}

In this section, we give a qualitative comparison between the coefficient of quasi-conformality $K$ and the Yamabe quotient $Q$ in the regime where $K$ is close to 1 and $Q$ is close to $Q(B)$. We will show that three types of geometric features which prevent $K$ from being close to 1 also prevent $Q$ from being close to $Q(B)$.

\subsection{Spikes}
Fix an angle $0 < \alpha < \frac{\pi}{2}$ and then let $C$
be a convex circular cone with opening angle $\alpha$. Gehring and V\"ais\"al\"a \cite{gehring1965coefficients} proved that 
\[
K(C) \ge (1+\cos \alpha)^{1/6}.
\]
Note that as $\alpha \to \frac{\pi}{2}$, the cone $C$ approaches a half-space and the lower bound degenerates to 1.   Gehring and V\"ais\"al\"a \cite{gehring1965coefficients} further showed that any domain containing a spike satisfies the same lower bound. 
More precisely, 
consider a domain $\Omega \subset \R^3$. Assume there is a neighborhood $U$ of a point $q\in \bd \Omega$ with $\Omega \cap U = C \cap U$ where $C$ is a convex circular cone with opening angle $\alpha$ with vertex at $q$. Then Gehring and V\"ais\"al\"a \cite{gehring1965coefficients} prove that $\Omega$ satisfies the same lower bound 
\begin{equation}
\label{gv1}
K(\Omega) \ge (1+\cos \alpha)^{1/6}.
\end{equation}
Thus spikes with small opening angle force $K(\Omega)$ to be significantly larger than 1. 

Our results show that a similar picture is true for $Q$. Indeed, note that the choice of $\frac{\pi}{4}$ as the separation angle in the definition of $\mathcal A(\Omega)$ was essentially arbitrary. All of the proofs still work replacing $\frac{\pi}{4}$ in the definition of $\mathcal A(\Omega)$ with any other fixed $\theta > 0$. In particular, we have the following variant of Proposition \ref{A-structure1}. 

\begin{prop}
\label{A-structure-2}
    For every $\eps > 0$ and every angle $\theta > 0$ there is a $\delta = \delta(\eps,\theta) > 0$ so that if $Q(B)-Q(\Omega) < \delta$ then the following property holds: whenever $B(x,r)\subset \Omega$ has two points $y,z \in \bd B(x,r) \cap \bd \Omega$ with $\angle yxz \ge \theta$, it follows that $\bd B(x,r)$ is contained in the $\eps r$-neighborhood of $\bd \Omega$. 
\end{prop}

Here $\delta(\eps,\theta)$ will degenerate to 0 as $\theta \to 0$. Now consider a domain $\Omega \subset \R^3$ with a spike as above. Again this means that there is a neighborhood $U$ of a point $q\in \bd \Omega$ with $\Omega\cap U = C\cap U$ where $C$ is a convex circular cone with opening angle $\alpha$ with vertex at $q$.  Then $\Omega$ contains a ball $B(x,r)$ with two points $y,z\in \bd B(x,r)\cap \bd \Omega$ satisfying $\angle yxz \ge \theta$ where $\theta = \theta(\alpha) \to 0$ as $\alpha \to \frac{\pi}{2}$. Moreover, by choosing the center of the ball close enough to the vertex of the spike, we can ensure that $\bd B(x,r)$ is not contained in the $\frac{r}{10}$ neighborhood of $\bd \Omega$. Hence we obtain \[
Q(B) - Q(\Omega) \ge \delta\left(\frac{1}{10},\theta(\alpha)\right),
\]   
which is analogous to (\ref{gv1}).  (Strictly speaking, $Q$ is only defined for smooth domains. So to be precise, we should really replace $\Omega$ by a smoothed out version of $\Omega$ here. Assuming the smoothing is done in a very small neighborhood of the vertex, this does not affect the above argument.)

\subsection{Ridges} Fix an angle $0 < \beta < \pi$ and let $D\subset \R^3$ be a dihedral wedge of angle $\beta$. Then Gehring and V\"ais\"al\"a \cite{gehring1965coefficients} prove that 
\[
K(D) \ge \left(\frac{\pi}{\beta}\right)^{1/2}. 
\]
Note that as $\beta\to \pi$ the wedge $D$ approaches a half-space and this lower bound degenerates to 1. Gehring and V\"ais\"al\"a further show that any domain containing a ridge modeled on $D$ satisfies the same lower bound. More precisely, suppose $\Omega\subset \R^3$ is a domain. Assume there is a neighborhood $U$ of a point $q\in \bd \Omega$ with $\Omega\cap U = D\cap U$ where $D$ is a dihedral wedge of angle $\beta$ whose edge passes through $q$. Then $\Omega$ satisfies the same lower bound 
\begin{equation}
    \label{gv2}
    K(\Omega) \ge \left(\frac \pi \beta\right)^{1/2}. 
\end{equation}
Thus ridges with small opening angle force $K(\Omega)$ to be much larger than 1. 

Our results show a similar behavior for $Q$. Indeed, consider a domain $\Omega\subset \R^3$ with a ridge as above. Then, as in the spike case, $\Omega$ contains a ball $B(x,r)$ with two points $y,z\in \bd B(x,r)\cap \bd \Omega$ satisfying $\angle yxz \ge \Theta$ where $\Theta = \Theta(\beta)\to 0$ as $\beta \to \pi$. By choosing the center point $x$ close enough to the edge of the ridge, we can ensure that $\bd B(x,r)$ is not contained in the $\frac r {10}$ neighborhood of $\bd \Omega$.  Therefore we obtain a bound 
\[
Q(B) - Q(\Omega) \ge \delta\left(\frac{1}{10}, \Theta(\beta)\right),
\]
where $\delta$ is as in Proposition \ref{A-structure-2}. This bound is analogous to (\ref{gv2}) in the Yamabe quotient setting. (Again, to be precise, we should really take a smoothing of $\Omega$ here. If the smoothing is done in a small neighborhood of the edge of the ridge, the above argument carries through unchanged.)

\subsection{Hair}

Let $\Omega \subset \R^3$ be a domain and assume there are concentric balls $B(x,a) \subset B(x,b)$ such that $\Omega \cap B(x,b)$ has two distinct connected components that intersect $B(x,a)$. Then Gehring and V\"ais\"al\"a \cite{gehring1965coefficients} prove that 
\[
K(\Omega) \ge A \log\left(\frac{b}{a}\right)
\]
where $A \ge 0.129$ is an absolute constant. In particular, this implies that a domain $\Omega$ with $K(\Omega)$ close to 1 cannot contain a long, thin ``hair'' which runs parallel to the domain. 
\[
\includegraphics[width=2.5in]{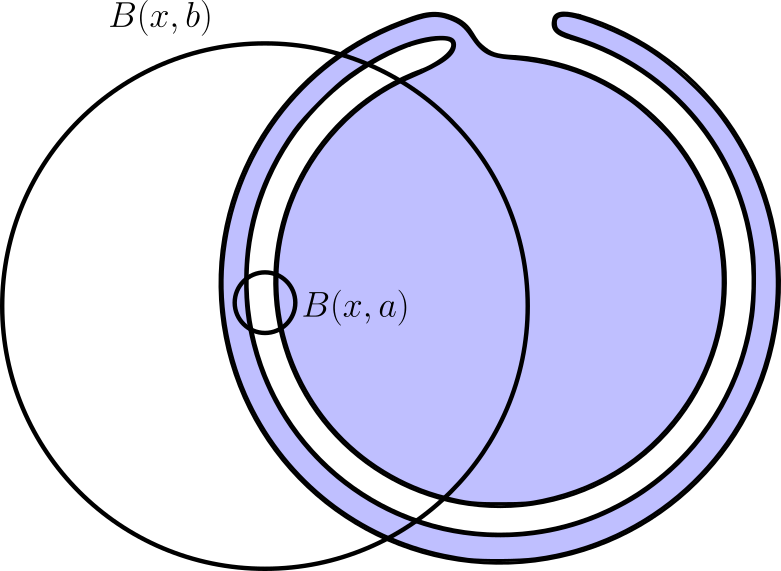}
\]
Again, our results show that a similar property is true for $Q$. Indeed, if $Q(\Omega)$ is close enough to $Q(B)$, then after scaling Theorem \ref{main1} implies that $B(0,1)\subset \Omega \subset B(0,1+\eps)$. Then Theorem \ref{main2} implies that $\Omega$ cannot contain any  ``hair'' in the region $B(0,1+\eps)\setminus B(0,1)$. Indeed, every point $x\in \Omega \setminus B(0,1)$ can be joined to a point of $B(0,1)$ by a path of length at most $\eps$  in $\Omega$. 

\bibliographystyle{plain}
\bibliography{bibliography.bib}

\end{document}